 \documentclass{article}
\usepackage{amsmath}
 \usepackage{amsfonts}
 \usepackage{amsthm}
 \usepackage{amssymb}
\usepackage{dsfont}\let\mathbb\mathds
\usepackage[english,frenchb]{babel}
 \usepackage[latin1]{inputenc}
\usepackage[T1]{fontenc}

\usepackage[all]{xy}

\usepackage[active]{srcltx}
\usepackage[breaklinks=true]{hyperref}
\usepackage{eprint}

\DeclareMathOperator{\MOD}{MOD}
\DeclareMathOperator{\Vect}{Vect}
\DeclareMathOperator{\PAR}{PAR}
\DeclareMathOperator{\Par}{Par}
\DeclareMathOperator{\FPar}{FPar}
\DeclareMathOperator{\EFPar}{EFPar}
\DeclareMathOperator{\Hom}{Hom}
\DeclareMathOperator{\Homb}{\bold{Hom}}
\DeclareMathOperator{\Ext}{Ext}
\DeclareMathOperator{\obj}{obj}
\DeclareMathOperator{\Tot}{Tot} 
\DeclareMathOperator{\SPEC}{\bold{Spec}}   
\DeclareMathOperator{\Sym}{Sym}
\DeclareMathOperator{\gr}{gr}
\DeclareMathOperator{\pr}{pr}
\DeclareMathOperator{\spec}{spec}
\DeclareMathOperator{\coker}{coker}
\DeclareMathOperator{\LC}{LC}
\DeclareMathOperator{\SLC}{SLC}
\DeclareMathOperator{\SLCF}{SLCF}
\DeclareMathOperator{\LCF}{LCF}
\DeclareMathOperator{\Champs}{Champs}
\DeclareMathOperator{\Rev}{Rev}
\DeclareMathOperator{\Cat}{Cat}
\DeclareMathOperator{\Ens}{Ens}
\DeclareMathOperator{\F}{F}
\DeclareMathOperator{\EF}{EF}
\DeclareMathOperator{\RH}{RH}
\DeclareMathOperator{\NS}{NS}
\DeclareMathOperator{\SS_0}{SS_0}
\DeclareMathOperator{\Rep}{Rep}
\DeclareMathOperator{\prof}{prof}
\DeclareMathOperator{\cone}{cone}
\DeclareMathOperator{\Aut}{Aut}
\DeclareMathOperator{\Ind}{Ind}
\DeclareMathOperator{\Pic}{Pic}
\DeclareMathOperator{\Pico}{Pic^0}
\DeclareMathOperator{\Pict}{Pic^{\tau}}
\DeclareMathOperator{\PIC}{\mathcal{P}\it{ic}}
\DeclareMathOperator{\SCPIC}{\bold{Pic}}
\DeclareMathOperator{\rg}{rg}

\newtheorem{thm}{Th\'eor\`eme}
\newtheorem{defi}{D\'efinition}
\newtheorem{cor}{Corollaire}

\newtheorem{lem}{Lemme}
\newtheorem{prop}{Proposition}
\newtheorem{rem}{Remarque}

  \newcommand{\UN}[4][r]{%
    \ar@/^1pc/[#1]^{#2}_*=<0.3pt>{}="HAUT"
    \ar@/_1pc/[#1]_{#3}^*=<0.3pt>{}="BAS"
    \ar @{=>} "HAUT";"BAS" ^{#4}
  }

\newcommand{\DEUX}[4][r]{%
    \ar@/^1pc/[#1]^{#2}_*=<0.3pt>{}="HAUT"
    \ar@/_1pc/[#1]_{#3}^*=<0.3pt>{}="BAS"
    \ar @{=>} "BAS";"HAUT" _{#4}
  }

\begin{document}
 
\bibliographystyle{tocplain} 
\author{Niels Borne}
\title{Sur les représentations du groupe fondamental d'une variété privée d'un diviseur à croisements normaux simples
}
 
 \maketitle

\section{Introduction}
\label{sec:introduction}

\subsection{Une description alternative}
\label{sec:descr:alter}

En l'absence de lacets, la recherche d'une description alternative du groupe fondamental étale (défini dans \cite{SGA1} en termes de revêtements) est une question classique, motivée essentiellement par la volonté de déterminer algébriquement des groupes fondamentaux qui ne sont connus que par voie transcendante. 

L'étude systématique du lien entre revêtements de, disons, une variété algébrique projective $X$, et certain fibrés sur $X$, commence avec Weil (\cite{Weil}). Celui-ci montre qu'un revêtement galoisien non ramifié de surfaces de Riemann $Y\rightarrow X$ permet d'associer à toute représentation $V$ complexe du groupe de Galois $G$ un fibré sur $X$ : on descend le fibré trivial $Y\times V$ sur $Y$ en $E=Y\times V/G$. Weil remarque que cette opération est compatible avec le produit tensoriel, ce qui confère des propriétés remarquables aux fibrés associés : ils sont en particulier \emph{finis}, au sens qu'il existe deux polynômes distincts $f,g$ à coefficients entiers positifs tels que $f(E)\simeq g(E)$. Il voit dans ces fibrés la généralisation des fibrés en droite de torsion, et commence à les caractériser. 

Ce travail trouve son aboutissement dans la formulation de Nori (\cite{NoriRFG}): la catégorie des fibrés finis sur $X$ est tannakienne, et le groupe de Tannaka associée est le groupe fondamental (profini) de $X$. Ceci a l'avantage d'être vérifié pour un schéma $X$ propre, réduit, connexe, sur un corps algébriquement clos de caractéristique $0$. En caractéristique $p$, le groupe de Tannaka de la catégorie des fibrés essentiellement finis (le schéma en groupe fondamental de Nori) se surjecte dans le groupe fondamental de $X$. Toutefois, comme le souligne Nori, cette description algébrique (les fibrés finis ne dépendant, en fait, que de la topologie de Zariski de $X$) du groupe fondamental n'a que peu d'utilité, puisque les fibrés finis de rang plus grand que $1$ semblent très difficiles à construire ex nihilo (i.e. sans utiliser de revêtement).

Partant du problème de la détermination algébrique du groupe fondamental, 
l'étude des courbes ouvertes (par exemple la droite projective moins trois points) apparaît plus abordable que celle des courbes complètes. En effet Nori montre dans \cite{NoriFGS} qu'il existe une équivalence de catégories tannakiennes entre la catégorie des représentations du groupe fondamental de la courbe ouverte et celles des fibrés \emph{paraboliques} (au sens de Seshadri, \cite{Seshadri}) finis. Il semble ardu, mais peut-être pas impossible, de construire algébriquement de tels fibrés.

Cependant, Nori ne fait qu'esquisser une preuve de cette équivalence. Cet article répond au souci d'en donner une démonstration complète et indépendante, suffisamment générale pour être valable en toute dimension. Plus précisément, on définit, donné un schéma $X$ propre, normal, connexe sur un corps $k$, et $\bold D=(D_i)_{i\in I}$ une famille de diviseurs irréductibles à croisements normaux simples sur $X$, la catégorie $\FPar(X,\bold D)$ (resp. $\EFPar(X,\bold D)$) des fibrés paraboliques (modérés) finis (resp. essentiellement finis), et notre résultat principal (théorème \ref{thmfinal}) s'énonce : 

\begin{thm}
	\label{thmfinal2}
 Soit  $D=\cup_{i\in I}D_i$, et $x\in X(k)$ un point rationnel, $x\notin D$.

  \begin{enumerate}
  \item[(i)] La paire $(\EFPar(X,\bold D),x^*)$ est une catégorie tannakienne.
  \item[(ii)] Si $k$ est algébriquement clos de caractéristique $0$, tout fibré parabolique essentiellement fini est fini, et le groupe de Tannaka de 
$(\FPar(X,\bold D),x^*)$ est canoniquement isomorphe au groupe fondamental $\pi_1(X-D,x)$.
 \end{enumerate}
 \end{thm}

Le premier point permet de proposer, en caractéristique positive, 
une définition du schéma en groupe fondamental modéré $\pi^D(X,x)$ comme groupe de Tannaka de la catégorie $(\EFPar(X,\bold D),x^*)$. Ce schéma en groupe est un hybride du schéma en groupe fondamental de Nori (\cite{NoriRFG}) et du groupe fondamental modéré de Grothendieck-Murre (\cite{GM}).

Un fait marquant est l'omniprésence de certains champs de Deligne-Mumford, les champs des racines, tout au long de cet article. Ils sont construits à partir de la paire $(X,\bold D)$ en ajoutant une structure d'orbifold le long des diviseurs, et sont en ce sens des ``schémas tordus''. Bien qu'il soient absents de l'énoncé de notre résultat principal, ils en sont absolument au coeur, leur présence éclairant d'un jour nouveau d'anciens problèmes. Par exemple le curieux produit dans la catégorie galoisienne des revêtements modérés de \cite{GM} s'avère être un produit fibré usuel sur un tel champ des racines\footnote{voir lemme \ref{norm} \it{(ii)}}.

Pour conclure, on propose une méthode de construction des fibrés (paraboliques) finis, inspirée de la méthode des petits groupes de Wigner et Mackey en théorie des représentations.

\subsection{Organisation de l'article}
\label{sec:organ-de-lart}

Dans la partie \ref{sec:fibr-parab}, on définit les fibrés paraboliques sur un schéma $X$ le long d'une famille régulière de diviseurs $\bold D=(D_i)_{i\in I}$, à poids rationnels à dénominateurs dans une famille d'entiers fixée $\bold r= (r_i)_{i\in I}$. Ceci se fait en deux temps : on définit d'abord (partie \ref{sec:definition}) les faisceaux paraboliques, puis (partie \ref{sec:fibres-paraboliques}) la notion de liberté locale pour un tel faisceau. On montre le caractère récursif de cette condition (proposition \ref{parrec}), qui se simplifie considérablement lorsque la famille est à croisements normaux simples (proposition \ref{parfamdivcroisnormsimpl}). On rappelle ensuite (partie \ref{sec:champ-des-racines}) la notion de champ des racines associé à la donnée de $X$, $\bold D$, et $\bold r$, et la partie se poursuit par le théorème \ref{fibparfibchamp} qui identifie les fibrés paraboliques aux fibrés usuels sur le champ des racines. 

Les deux parties suivantes sont de nature plus technique.

Le résultat essentiel de la partie \ref{sec:groupe-fond-modere} est la proposition qui donne une interprétation du groupe fondamental modéré comme limite projective de groupes fondamentaux de champs des racines. C'est une conséquence à peu près immédiate du lemme d'Abhyankar.

La partie \ref{sec:fibres-finis} étudie le lien entre groupe fondamental d'un champ de Deligne-Mumford convenable et fibrés finis. Ce lien s'exprime par une équivalence de catégories tannakiennes entre systèmes locaux $k$-vectoriels de rang fini sur le champ et fibrés finis donné par un foncteur ``à la Riemann-Hilbert'', voir le corollaire \ref{coreftan}. Le point crucial (théorème \ref{eftan}) est le fait que les fibrés essentiellement finis sur les champs des racines (et un peu plus généralement sur des schémas tordus) propres et réduits sur un corps forment une catégorie tannakienne. 

La partie suivante (partie \ref{sec:theoreme-de-weil}) est une partie de synthèse où l'on assemble les différents éléments pour aboutir au théorème \ref{thmfinal2}.

Enfin la dernière partie (partie \ref{calholres}) est consacrée à l'application du théorème \ref{thmfinal2} au calcul explicite de fibrés paraboliques finis de groupe d'holonomie résoluble.

\subsection{Origines et liens avec des travaux existants}
\label{sec:origines}

La définition des fibrés paraboliques par rapport à une famille régulière de diviseurs $\bold D$ dans la partie \ref{sec:fibr-parab} est inspirée de celle de Maruyama-Yokogawa \cite{MY}. Une différence importante avec ces auteurs est l'emploi d'indices multiples (moralement, autant d'indices que de composantes irréductibles régulières de $\bold D$) ce qui mène naturellement à la définition \ref{faiscpar}, essentiellement équivalente à celle employée par Iyer et Simpson, voir \cite{IyerSimpson}. Ceci complique singulièrement la condition de liberté locale pour un faisceau parabolique : la définition \ref{defpar} en termes d'homologie de complexes associés à des facettes est entièrement originale et semble apporter un éclairage nouveau sur la notion de fibré parabolique ``localement abélien`` employée dans \cite{IyerSimpson} (voir la remarque \ref{locab}).

Les champs des racines ont été introduits par Vistoli (\cite{AGV}) et Cadman (\cite{Cadman}). L'identification des fibrés paraboliques avec les fibrés sur les champs des racines a été initiée dans \cite{Borne} dans la situation à indice unique, sa généralisation à la situation présente ne pose pas de problème particulier. On peut trouver certains précurseurs de ce résultat, en particulier dans le travail de Biswas (\cite{Biswasorb}, \cite{BBN}), mais ces auteurs n'employant que des fibrés paraboliques à indice unique, ces résultats ne nous semblent corrects que dans le cas d'un diviseur régulier.

La preuve donnée du théorème de Nori parabolique (théorème \ref{thmfinal}) est complètement indépendante de la preuve de Nori en dimension $1$ (Nori n'utilisant pas la définition de Seshadri des fibrés paraboliques), mais, grâce à l'utilisation des champs des racines, suit d'assez près la démonstration de la version classique (non parabolique) du théorème : en particulier la démonstration que les fibrés essentiellement finis sur un ``schéma tordu`` forment une catégorie tannakienne (théorème \ref{eftan}) est une adaptation directe de la preuve que Nori donne pour un schéma usuel. Toutefois, l'emploi d'un foncteur de type Riemann-Hilbert (voir \S \ref{sec:foncteur-la-riemann}) pour faire le lien entre système locaux et fibrés finis, bien que naturelle, semble nouvelle dans ce contexte. Les travaux de Grothendieck-Murre (\cite{GM}) sur le groupe fondamental modéré, ainsi que ceux de Noohi (\cite{Noohi}) et Zoonekynd (\cite{Zoon}) sur le groupe fondamental des champs de Deligne-Mumford, sont d'autres ingrédients importants de la preuve.

Élémentaire, mais à priori un peu surprenante, l'idée de construire des fibrés finis par image directe (voir la proposition \ref{imfibrfin} et le lemme \ref{holres}) ne semble pas avoir été déjà employée.

Enfin, pour conclure sur les insuffisances de cet article, il serait naturellement souhaitable d'avoir une version du théorème de Nori parabolique pour un diviseur à croisements normaux généraux, ou sur un corps quelconque.

\subsection{Remerciements}
\label{sec:remerciements}

Ce travail doit beaucoup à Angelo Vistoli, ses contours n'étant apparus nettement qu'à la suite d'une visite à Bologne en janvier 2006. Je l'en remercie chaleureusement. Je tiens également à remercier Alessandro Chiodo,  Michel Emsalem, Boas Erez, Madhav Nori, Martin Olsson et Gabriele Vezzosi pour d'intéressantes discussions sur le sujet.

\section{Fibrés paraboliques le long d'une famille régulière de diviseurs}
\label{sec:fibr-parab}

Dans cette partie, on notera $X$ un champ de Deligne-Mumford localement noethérien, $I$ un ensemble fini, $\bold D=(D_i)_{i\in I}$ une famille de diviseurs de Cartier effectifs sur $X$, $\bold r=(r_i)_{i\in I}$ une famille d'entiers $r_i\geq 1$.
\subsection{Faisceaux paraboliques}
\label{sec:faisc-parab}

\subsubsection{Définition}
\label{sec:definition}

Soient d'abord $\mathcal I,\mathcal C$ deux catégories monoïdales, $\mathcal I$ étant supposée stricte.
Un foncteur monoïdal $F:\mathcal I\rightarrow \mathcal C$ permet de voir $\mathcal C$ comme un $\mathcal I$-module (relâché) sur le monoïde $\mathcal I$ via l'opération

\xymatrix@R=2pt{
&&\mathcal I\times \mathcal C \ar[r]& \mathcal C\\
&&(I,C)  \ar[r] & F(I)\otimes C
}

On considère à présent $\mathcal I=(\mathbb Z ^I)^{op}$,
$\mathcal C_1=(\frac{1}{\bold r}\mathbb Z^I)^{op}$, où  par définition 
$\frac{1}{\bold r}\mathbb Z^I=\prod_{i\in I}\frac{1}{r_i}\mathbb Z$, 
et $\mathcal C_2=\MOD X$, la catégorie des faisceaux de $\mathcal O_X$-modules sur $X$. $\mathcal I$ et $\mathcal C_1$ sont vues comme catégories associées aux ensembles ordonnés correspondants, et munies du produit tensoriel induit par l'addition, quant à $\mathcal C_2$, elle est munie de sa structure monoïdale canonique. Enfin, on dispose du foncteur d'inclusion
$F_1: \mathcal I \rightarrow \mathcal C_1$ et du foncteur  

\xymatrix@R=2pt{
&&F_2: \mathcal I \ar[r] &\mathcal C_2\\
&&\bold l=(l_i)_{i\in I}\ar[r] & \mathcal O_X(-\bold l\bold D)=\mathcal O_X(-\sum_{i\in I}l_iD_i)}
qui permettent de voir $\mathcal C_1$ et $\mathcal C_2$ comme des 
$\mathcal I$-monoïdes. 

\begin{defi}
	\label{faiscpar}
  On définit la \emph{catégorie des faisceaux paraboliques} sur $X$ le long de $\bold D$ à poids multiples de $\frac{1}{\bold r}$ comme la catégorie des morphismes de modules sur le monoïde $(\mathbb Z ^I)^{op}$  :

$$\PAR_{\frac{1}{\bold r}}(X,\bold D)= \Hom_{(\mathbb Z ^I)^{op}}((\frac{1}{\bold r}\mathbb Z^I)^{op}, \MOD X)$$ 
\end{defi}

Plus en détail, un objet de $\PAR_{\frac{1}{\bold r}}(X,\bold D)$ est un 
couple $(\mathcal E_\cdot, j)$, où $\mathcal E_\cdot : (\frac{1}{\bold r}\mathbb Z^I)^{op}\rightarrow \MOD X$ est un foncteur (non nécessairement monoïdal !), et
$j$ est un isomorphisme naturel (dit \emph{isomorphisme des pseudo-périodes}):

\xymatrix@R=12pt{
  &&&&  (\mathbb Z ^I)^{op} \times (\frac{1}{\bold r}\mathbb Z^I)^{op} \ar[rr]^{+} \ar[dd]_{(\mathbb Z ^I)^{op} \times\mathcal E_\cdot} && (\frac{1}{\bold r}\mathbb Z^I)^{op}
   \ar[dd]^{\mathcal E_\cdot}\\
  &&&&&&\\
  &&&& (\mathbb Z ^I)^{op} \times \MOD X \ar[rr]_{\mathcal O_X(-\cdot \bold D)\otimes \cdot} \ar@{=>}[uurr]^{j}&& \MOD X}

On omettra souvent $j$ pour alléger les notations. Un morphisme $(\mathcal E_\cdot, j)\rightarrow (\mathcal E'_\cdot, j')$ est une transformation naturelle :

\xymatrix{\relax
    &&&&&(\frac{1}{\bold r}\mathbb Z^I)^{op} \UN[r]{\mathcal E_\cdot}{\mathcal E'_\cdot}{\alpha} & \MOD X }

  compatible avec $j$ et $j'$, en un sens évident.

\subsubsection{Opérations élémentaires sur les faisceaux paraboliques}
\label{sec:oper-elem-sur}

Pour $\bold l\in \obj \frac{1}{\bold r}\mathbb Z^I$ et $\mathcal E_\cdot \in \obj \PAR_{\frac{1}{\bold r}}(X,\bold D)$ on dispose du \emph{décalage} défini de la manière usuelle 

\xymatrix@R=2pt{&&&
\mathcal E_\cdot[\bold l]:(\frac{1}{\bold r}\mathbb Z^I)^{op}\ar[r]^{+\bold l}&(\frac{1}{\bold r}\mathbb Z^I)^{op}\ar[r]^{\mathcal E_\cdot}&\MOD X}

l'isomorphisme des pseudo-périodes étant induit par celui de $\mathcal E_\cdot$ de la manière évidente.

Passons au \emph{produit tensoriel des faisceaux paraboliques} :
donnés $\mathcal E_\cdot, \mathcal E'_\cdot \in \obj \PAR_{\frac{1}{\bold
    r}}(X,\bold D)$, on dispose, pour tout $\bold l\in \obj \frac{1}{\bold
  r}\mathbb Z^I$, de l'objet $\mathcal E_{\bold l}\otimes \mathcal
E'_\cdot[-\bold l]$ de $\PAR_{\frac{1}{\bold r}}(X,\bold D)$ obtenu comme
produit tensoriel externe de $\mathcal E_{\bold l}\in\obj\MOD X$ par
$\mathcal E'_\cdot[-\bold l]\in\obj\PAR_{\frac{1}{\bold r}}(X,\bold D)$.

Cette quantité étant dinaturelle en $\bold l$, on peut définir $(\mathcal
E_\cdot\otimes\mathcal E'_\cdot)_\cdot$  comme la cofin\footnote{pour des
  détails sur la notion de cofin (coend), voir \cite{MacLane}, ou bien
  \cite{Borne}, Appendice B} du foncteur de
variance mixte correspondant dans $\PAR_{\frac{1}{\bold r}}(X,\bold D)$ (qui est cocomplète vu que $\MOD X$ l'est) : 

$$(\mathcal E_\cdot\otimes\mathcal E'_\cdot)_\cdot=\int^{\bold l\in {\frac{1}{\bold
  r}\mathbb Z^I}} \mathcal E_{\bold l}\otimes \mathcal
E'_\cdot[-\bold l]$$

Ce produit tensoriel est l'adjoint à gauche (enrichi) du foncteur $\Homb$
interne naturel de $\PAR_{\frac{1}{\bold r}}(X,\bold D)$ (dont nous n'aurons pas usage).

Passons au \emph{structures spéciales} : l'inclusion 
$\mathbb Z^I\rightarrow \frac{1}{\bold r}\mathbb Z^I$ admet pour adjoint à gauche le foncteur 

\xymatrix@R=2pt{&&&&& \frac{1}{\bold r}\mathbb Z^I\ar[r]&\mathbb Z^I\\
&&&&&\bold l\ar[r] &-[-\bold l]}

où $[\bold l]=([l_i])_{i\in I}$, $[\cdot]$ désignant la partie entière (on espère que cela n'entraînera pas de confusion avec la notation du décalage).

On en déduit que le foncteur d'oubli (ou d'évaluation en zéro)

$\Hom_{(\mathbb Z ^I)^{op}}((\frac{1}{\bold r}\mathbb Z^I)^{op}, \MOD X)
\rightarrow \Hom_{(\mathbb Z ^I)^{op}}((\mathbb Z^I)^{op}, \MOD X)\simeq \MOD X$

admet un adjoint à gauche, qu'on notera $\mathcal E\rightarrow \underline{\mathcal E}_\cdot$, défini par 

\begin{equation}
\label{strspe}
\underline{\mathcal E}_\cdot=\mathcal E\otimes \mathcal O_X([-\cdot]\bold D)
\end{equation}

On appellera $\underline{\mathcal E}_\cdot$ \emph{le faisceau parabolique à structure spéciale} induit par  $\mathcal E$.

Lorsque $\bold D=\bold r \bold E$, on dispose d'un faisceau parabolique particulier, défini comme foncteur par 

\begin{equation}
\label{parpart}
\bold l \rightarrow \mathcal O_X(-\bold l\bold r\bold E)
\end{equation}

l'isomorphisme des pseudo-périodes étant défini de la manière évidente. On le
notera simplement $\mathcal O_X(-\cdot\bold r\bold E)$.

Enfin il est clair que si $f:X'\rightarrow X$ est un morphisme plat,
la paire de foncteurs adjoints $(f^*,f_*)$ entre les catégories $\MOD X$ et $\MOD X'$ induit une adjonction similaire entre les catégories $\PAR_{\frac{1}{\bold r}}(X,\bold D)$ et $\PAR_{\frac{1}{\bold r}}(X',f^*\bold D)$.

\subsection{Fibrés paraboliques}
\label{sec:fibres-paraboliques}

\subsubsection{Facette}
\label{sec:facette}

On va définir certaines applications dont le but est $\frac{1}{\bold r}\mathbb Z^I=\prod_{i\in I}\frac{1}{r_i}\mathbb Z$, vu comme ensemble ordonné.
\begin{defi}
\label{facettedef}
  Soit $J\subset I$ un sous-ensemble, et $(e_i)_{i\in I}$ la base canonique de $\mathbb Z^I$. On appelle \emph{facette} toute application croissante 

$$F:\{0<1\}^J\rightarrow \frac{1}{\bold r}\mathbb Z^I$$ qui est affine au sens suivant : il existe une famille d'entiers $\boldsymbol{\epsilon}=(\epsilon_i)_{i\in J}$ vérifiant $\forall i\in J\; \;
1\leq \epsilon_i\leq r_i$ telle que :

$$\forall \boldsymbol{\mu}=(\mu_i)_{i\in J}\in  \{0<1\}^J\;\;
F(\boldsymbol{\mu})=F(\bold 0)+\sum_{i\in J} \mu_i\frac{\epsilon_i}{r_i}e_i$$ 
\end{defi}

Par la suite, on identifiera une facette entre deux ensembles ordonnés 
avec le foncteur entre les catégories correspondantes. 

La donnée d'une facette équivaut bien entendu à celle de $J$, $F(\bold 0)$, et
de la famille $\boldsymbol{\epsilon}$. Dans le cas particulier où $F(\bold
0)=\bold 0$ et $\boldsymbol{\epsilon}=\bold{r}_{|J}$, on parlera de la
\emph{facette spéciale} associée à $J$, et on la notera $F_J$, il s'agit
simplement de ``l'inclusion'' $F_J:\{0<1\}^J\rightarrow \frac{1}{\bold
  r}\mathbb Z^I$. Plus généralement, lorsque $F(\bold 0)=\bold 0$, on notera
$F_{\frac{\boldsymbol{\epsilon}}{\bold r}J}$ la facette correspondante.

\subsubsection{Complexe associé à un faisceau parabolique et une facette}
\label{sec:complexe-associe}

On commencer par préciser les conventions  utilisées concernant les \emph{complexes multiples} (essentiellement empruntées à \cite{Weibel}). Soit $J$ un ensemble fini.
Un complexe (de chaînes) multiple à valeurs dans une catégorie abélienne $\mathcal A$ est un foncteur $C.=(\mathbb N^J)^{op} \rightarrow \mathcal A$ tel que si on note $(e_i)_{i\in J}$ la base canonique de $\mathbb N^J $, alors pour tout multi-indice $\boldsymbol{\alpha}=(\alpha_i)_{i\in J}$, et pour tout $j\in J$, les morphismes 
$d^j_{\boldsymbol{\alpha}}:C_{\boldsymbol{\alpha}}\rightarrow C_{\boldsymbol{\alpha}-e_j}$
vérifient $d^j_{\boldsymbol{\alpha}-e_j}\circ d^j_{\boldsymbol{\alpha}}=0$.

A un tel complexe, on associe de la manière son \emph{complexe total} $\Tot(C_\cdot)_\cdot$ de la manière usuelle : on commence par transformer $C_\cdot$ en complexe anti-commutatif en fixant un ordre total arbitraire sur l'ensemble $J$ et en modifiant les différentielles de la manière suivante : on pose pour tout multi-indice $\boldsymbol{\alpha}=(\alpha_i)_{i\in J}$ et pour tout $j\in J$ :

$$\delta^j_{\boldsymbol{\alpha}}=(-1)^{\sum_{i<j}\alpha_i}d^j_{\boldsymbol{\alpha}}$$

L'appellation complexe anti-commutatif est justifiée par le fait qu'alors 
$\delta^j\delta^{j'}=-\delta^{j'}\delta^j$, lorsque ces expressions ont un sens, 
pour tout $j$, $j'$ dans $J$. Le complexe total est alors défini par, pour $n\geq 0$ : 

 $$\Tot(C_\cdot)_n =\oplus_{\boldsymbol{\alpha},|\boldsymbol{\alpha}| =n}C_{\boldsymbol{\alpha}}$$

où $|\boldsymbol{\alpha}|= \sum_{i\in J}\alpha_i$, 
et les différentielles étant définies par
$\delta_n=\oplus_{\boldsymbol{\alpha},|\boldsymbol{\alpha}| =n}\sum_{i\in
  J}\delta^i_{\boldsymbol{\alpha}}$. Le fait qu'on parte d'un complexe
anti-commutatif assure qu'on obtient bien ainsi un complexe (simple).

\begin{defi}
Soient  $\mathcal E_\cdot \in\obj\PAR_{\frac{1}{\bold r}}(X,\bold D)$ un faisceau parabolique et $F:\{0<1\}^J\rightarrow \frac{1}{\bold r}\mathbb Z^I$ une facette. 

On appelle \emph{complexe multiple associé} le prolongement par zéro du foncteur composé $\mathcal E_\cdot\circ F^{op}$ à $(\mathbb N^J)^{op}$, comme dans le diagramme ci-dessous :

\xymatrix@R=2pt{
&&&&(\mathbb N^J)^{op}\ar@{.>}[r]&\MOD X\\
&&&&& \\
&&&&& \\
&&&&(\{0<1\}^J)^{op} \ar@{^{(}->}[uuu]\ar[uuur]_{\mathcal E_\cdot\circ F^{op}}&}
Par abus, on notera encore ce complexe multiple $\mathcal E_\cdot\circ F^{op}$.
\end{defi}

Le sens de ces définitions apparaît dans le cas particulier où $ \mathcal E_\cdot=\underline{\mathcal O_X}_\cdot$, le faisceau structurel muni de la structure parabolique spéciale, et $F=F_J$, la facette spéciale associée à $J$.

\begin{lem}
\label{Koszul}
Pour $i\in I$ on note $(\mathcal O_X(-D_i)\rightarrow \mathcal O_X)$ le complexe de chaînes (simple) concentré en degrés $1$ et $0$. Pour tout $J\subset I$, il existe un isomorphisme naturel de complexes multiples :

$$\underline{\mathcal O_X}_\cdot\circ F_J^{op}\simeq \otimes_{i\in J}(\mathcal O_X(-D_i)\rightarrow \mathcal O_X)$$
\end{lem}

\begin{proof}
  Cela résulte simplement de l'expression de 
$\underline{\mathcal O_X}_\cdot$ donnée par (\ref{strspe}),\S \ref{sec:oper-elem-sur}.
\end{proof}

Cet exemple est crucial pour définir les fibrés paraboliques, puisque ceux-ci seront localement somme directe finie de faisceaux décalés du faisceau structurel, i.e. du type $\underline{\mathcal O_X}_\cdot[\bold l]$, pour $\bold l\in \obj \frac{1}{\bold r}\mathbb Z^I$ (voir remarque \ref{locab}), ce qui impose des contraintes fortes aux complexes associés à chaque facette.  

De manière similaire, on a :

\begin{lem}
\label{Koszul2} 
Supposons $\bold D=\bold r \bold E$.
Pour tout $J\subset I$, et tout $\boldsymbol{\epsilon}=(\epsilon_j)_{j\in J}$
comme dans \S \ref{sec:facette}, il existe un isomorphisme naturel de complexes multiples :

$$\mathcal O_X(-\cdot\bold r\bold E)\circ F_{\frac{\boldsymbol{\epsilon}}{\bold r}J}^{op}\simeq \otimes_{i\in J}(\mathcal O_X(-\epsilon_i E_i)\rightarrow \mathcal O_X)$$
\end{lem}

\begin{proof}
 Conséquence directe de la définition \ref{parpart}.
\end{proof}

\subsubsection{Définition des fibrés paraboliques}
\label{sec:defin-des-fibres}

\begin{lem}
\label{support}
  Soit $\mathcal E_\cdot \in\obj\PAR_{\frac{1}{\bold r}}(X,\bold D)$ un faisceau parabolique, et $F:\{0<1\}^J\rightarrow \frac{1}{\bold r}\mathbb Z^I$ une facette.
Il existe un morphisme naturel de complexes multiples

$$\mathcal E_{ F(\bold 0)}\otimes(\underline{\mathcal O_X}_\cdot\circ
F_J^{op})\rightarrow \mathcal E_\cdot\circ F^{op}$$ qui est un isomorphisme en
(multi-)degré $\bold 0$. Supposons de plus $\mathcal E_{ F(\bold 0)}$ localement libre de rang fini, alors si $i_J:\cap_{i\in J}D_i\rightarrow X$ désigne l'immersion fermée canonique, il existe une surjection naturelle 
${i_J}_*i_J^*\mathcal E_{ F(\bold 0)}\twoheadrightarrow H_0(\Tot(\mathcal E_\cdot\circ F^{op}))$.

\end{lem}

On repousse la preuve après la définition suivante. La deuxième assertion
montre que, sous les hypothèses du lemme \ref{support}, $H_0(\Tot(\mathcal
E_\cdot\circ F^{op}))$ est à support dans $\cap_{i\in J}D_i$, ce qui donne un
sens\footnote{cette définition dépend \emph{a priori} du choix d'un ordre
  total sur $I$, mais la proposition \ref{parrec} et la remarque
  \ref{parrecrem} qui s'ensuit montrent qu'il n'en est rien, au moins pour une famille régulière de
  diviseurs (voir \S\ref{sec:famille-reguliere}) qui est le seul cas que nous considérerons} à la :

\begin{defi}
\label{defpar}
Soit $\mathcal E_\cdot \in\obj\PAR_{\frac{1}{\bold r}}(X,\bold D)$ un faisceau parabolique. On dit que c'est un \emph{faisceau parabolique localement libre} ou encore un \emph{fibré parabolique} si pour toute facette $F:\{0<1\}^J\rightarrow \frac{1}{\bold r}\mathbb Z^I$, l'homologie $H_l(\Tot(\mathcal E_\cdot\circ F^{op}))$ du complexe multiple associé est nulle pour $l>0$, et est un faisceau localement libre de rang fini sur $\cap_{i\in J} D_i$ pour $l=0$. On notera $\Par_{\frac{1}{\bold r}}(X,\bold D)$ la catégorie des faisceaux paraboliques localement libres sur $X$ le long de $\bold D$ à poids multiples de $\frac{1}{\bold r}$.
\end{defi}

\begin{proof}[Démonstration du lemme \ref{support}]

La deuxième assertion est une conséquence de la première, puisque celle-ci entraîne l'existence d'un morphisme de complexes simples 
$\Tot(\mathcal E_{ F(\bold 0)}\otimes(\underline{\mathcal O_X}_\cdot\circ F_J^{op}))\rightarrow \Tot(\mathcal E_\cdot\circ F^{op})$ qui est un isomorphisme en degré $0$, d'où un épimorphisme 
$H_0(\Tot(\mathcal E_{ F(\bold 0)}\otimes(\underline{\mathcal O_X}_\cdot\circ F_J^{op})))\twoheadrightarrow H_0(\Tot(\mathcal E_\cdot\circ F^{op}))$. 
L'hypothèse sur $\mathcal E_{ F(\bold 0)}$ donne alors $H_0(\Tot(\mathcal E_{ F(\bold 0)}\otimes(\underline{\mathcal O_X}_\cdot\circ F_J^{op})))\simeq\mathcal E_{ F(\bold 0)}\otimes H_0(\Tot((\underline{\mathcal O_X}_\cdot\circ F_J^{op})))$.
Le lemme \ref{Koszul} montre que $\Tot((\underline{\mathcal O_X}_\cdot\circ F_J^{op}))$ est le complexe de Koszul associé à la famille $(D_i)_{i\in J}$, ce qui permet de conclure.

Reste à montrer la première assertion. L'hypothèse que $F$ est une facette
montre qu'on a en particulier $\forall \boldsymbol{\mu}=(\mu_i)_{i\in J}\in
\{0<1\}^J\;\; F(\boldsymbol{\mu})\leq F(\bold 0)+F_J(\boldsymbol{\mu})
$. Cette inégalité entre applications croissantes peut s'interpréter comme l'existence d'une transformation naturelle $\leq$ entre les foncteurs correspondants :

\xymatrix{\relax
    &&&&&\{0<1\}^J \UN[r]{F}{F(\bold 0)+F_J}{\leq} & \frac{1}{\bold r}\mathbb Z^I}

En passant aux catégories opposées et en composant avec $\mathcal E_\cdot$

\xymatrix{\relax
    &&&&({\{0<1\}^J})^{op} \DEUX[r]{{F}^{op}}{{(F(\bold 0)+F_J)}^{op}}{{\leq}^{op}} & ({\frac{1}{\bold r}\mathbb Z^I})^{op}\ar[r]^{\mathcal E_\cdot}&\MOD X}

on obtient une transformation naturelle :
$\mathcal E_\cdot\circ{\leq}^{op}:\mathcal E_\cdot\circ {(F(\bold 0)+F_J)}^{op}\rightarrow \mathcal E_\cdot\circ {F}^{op}$. 

Le diagramme suivant 

\xymatrix@R=12pt{
&&({\{0<1\}^J)}^{op}\ar[dr]|-{({F_J}^{op},F(\bold 0))}\ar@/^/[drrr]^{{(F(\bold 0)+F_J)}^{op}}\ar@/_/[dddr]_{({F_J}^{op},\mathcal E_{F(\bold 0)})}&&&\\
  &&&  (\mathbb Z ^I)^{op} \times (\frac{1}{\bold r}\mathbb Z^I)^{op} \ar[rr]^{+} \ar[dd]|-{(\mathbb Z ^I)^{op} \times\mathcal E_\cdot} && (\frac{1}{\bold r}\mathbb Z^I)^{op}
   \ar[dd]^{\mathcal E_\cdot}\\
  &&&&&\\
  &&& (\mathbb Z ^I)^{op} \times \MOD X \ar[rr]_{\mathcal O_X(-\cdot \bold D)\otimes \cdot} \ar@{=>}[uurr]^{j}&& \MOD X}

montre que l'isomorphisme de pseudo-périodes $j$ de $\mathcal E_\cdot$ permet d'identifier $\mathcal E_\cdot\circ {(F(\bold 0)+F_J)}^{op}$ avec $\mathcal E_{ F(\bold 0)}\otimes(\underline{\mathcal O_X}_\cdot\circ F_J^{op})$, ce qui conclut la démonstration du lemme.
 
\end{proof}

\subsection{Fibrés paraboliques et revêtements}
\label{sec:fibr-parab-rev}

Pour montrer la pertinence de la définition \ref{defpar}, on doit montrer l'existence de fibrés paraboliques non triviaux, i.e. autre que les sommes directes de décalés de fibrés à structure spéciale (bien qu'ils soient tous localement de ce type). La manière la plus directe de produire de tels fibrés paraboliques est d'utiliser des revêtements de Kummer ramifiés le long d'une famille régulière de diviseurs. On commence par préciser ces notions.

\subsubsection{Famille régulière de diviseurs}
\label{sec:famille-reguliere}

On rappelle le lemme folklorique suivant, ainsi qu'une preuve, repoussée après
la définition \ref{famregdef}, à laquelle il donne un sens.

\begin{lem}
\label{famreglem}
  Soit $X$ un schéma localement noethérien, $I$ un ensemble fini, $\bold D=(D_i)_{i\in I}$ un ensemble de diviseurs de Cartier effectifs sur $X$, et pour tout $i\in I$, $s_i:\mathcal O_X\rightarrow \mathcal O_X(D_i)$ la section canonique.  Les conditions suivantes sont équivalentes :

  \begin{enumerate}

  \item[(i)] En tout point $x$ de $X$, soit une des sections $s_i$ est
    inversible, soit $(s_i)_{i\in I}$ est une suite \footnote{dans un anneau
    local noethérien, la régularité d'une suite ne dépend pas de l'ordre de ses éléments, ce qui permet de parler de famille régulière} 
régulière (au sens de Serre),
\item[(ii)] la section $(s_i)_{i\in I} :  \mathcal O_X\rightarrow\oplus_{i\in I}\mathcal O_X(D_i)$ est régulière (i.e. le complexe de Koszul associé au morphisme dual n'a pas d'homologie en degré supérieur ou égal à $1$), 
\item[(iii)] $\cap_ {i\in I} D_i\rightarrow X$ est une immersion fermée régulière, et si $\mathcal I_I$ est l'idéal engendré par les $(s_i)_{i\in I}$, alors en tout point $x$ de $\cap_ {i\in I} D_i$, les $(s_i)_{i\in I}$ forment un système minimal (pour le cardinal) de générateurs de $\mathcal I_{I,x}$.
  \end{enumerate}
\end{lem}

\begin{defi}
\label{famregdef}
  On dira que $\bold D=(D_i)_{i\in I}$ est \emph{une famille régulière de diviseurs} si pour tout sous-ensemble $J\subset I$, la sous-famille $\bold D_J=(D_i)_{i\in J}$ vérifie les conditions équivalentes du lemme \ref{famreglem}.
\end{defi}

\begin{proof}[Démonstration du lemme \ref{famreglem}]
 
 $(i)\implies(ii)$ : ceci résulte du fait que le complexe de Koszul associé soit à un morphisme surjectif, soit à une suite régulière, sont sans homologie en degré plus grand que $1$.

 $(ii)\implies(iii)$ : le fait que $\mathcal I_I$ soit régulier résulte directement de la définition (\cite{SGA6}, exposé VII, Définition 1.4). De plus en notant $\mathcal E_I= \oplus_{i\in I}\mathcal O_X(D_i)$, et $i_I:\cap_ {i\in I} D_i\rightarrow X$ l'immersion fermée canonique, la régularité de $s_I=(s_i)_{i\in I}$ montre que $\mathcal I_I/\mathcal I_I^2\simeq i_I^* \mathcal E_I^{\vee}$, donc en un point $x$ de $\cap_ {i\in I} D_i$, $\mathcal I_{I,x}/\mathcal I_{I,x}^2$ est libre de rang $\# I$, par conséquent $\mathcal I_{I,x}$ ne saurait être engendré par moins de $\# I$ éléments.

 $(iii)\implies(i)$ $\cap_ {i\in I} D_i\rightarrow X$ est une immersion fermée régulière donc en particulier quasi-régulière (\cite{SGA6}, exposé VII, Proposition 1.3). Autrement dit $\mathcal I_I/\mathcal I_I^2$ est localement libre de rang fini sur $\cap_ {i\in I} D_i$ et l'homomorphisme surjectif canonique 

$$\Sym_{\mathcal O_X/\mathcal I_I}(\mathcal I_I/\mathcal I_I^2)\twoheadrightarrow \gr_{\mathcal I_I}(\mathcal O_X)$$

est un isomorphisme. En un point $x$ de $\cap_ {i\in I} D_i$, le lemme de
Nakayama montre que le rang de $\mathcal I_{I,x}/\mathcal I_{I,x}^2$ sur
$\mathcal O_{X,x}/\mathcal I_{I,x}$ est le nombre minimal de générateurs de
$\mathcal I_{I,x}$ sur $\mathcal O_{X,x}$, à savoir $\#I$. Donc $(s_i)_{i\in
  I}$ est une base de  $\mathcal I_{I,x}/\mathcal I_{I,x}^2$ et par conséquent
les $s_i$ définissent un isomorphisme $\mathcal O_{X,x}/\mathcal I_{I,x}[(S_i)_{i\in I}]\simeq\Sym_{\mathcal O_{X,x}/\mathcal I_{I,x}}(\mathcal I_{I,x}/\mathcal I_{I,x}^2) $, où les $S_i$ sont des indéterminées. On en conclut que 
l'homomorphisme surjectif $\mathcal O_{X,x}/\mathcal I_{I,x}[(S_i)_{i\in I}]\twoheadrightarrow \gr_{\mathcal I_{I,x}}(\mathcal O_{X,x})$ défini par les $(s_i)_{i\in I}$ est un isomorphisme, autrement dit la famille $(s_i)_{i\in I}$ est quasi-régulière au sens de \cite{EGA4.1}, ${\bold 0}$ Définition 15.1.7, donc régulière (\cite{EGA4.1}, ${\bold 0}$ Corollaire 15.1.11).

\end{proof}

\begin{lem}
\label{famregprop}
  \begin{enumerate}
\item[(i)] Toute sous-famille d'une famille régulière l'est également.
  \item[(ii)] Si $(l_i)_{i\in I}$ est un ensemble d'entiers $l_i\geq 1$, alors la famille $(D_i)_{i\in I}$ est régulière si et seulement si la famille $(l_iD_i)_{i\in I}$ l'est.
 \end{enumerate}
\end{lem}

\begin{proof} 

 $(i)$ est immédiat. 

$(ii)$ résulte, par récurrence, de la caractérisation $(i)$ du lemme
\ref{famreglem}, puisque qu'un élément est diviseur de zéro (resp. inversible) si et seulement une de ses puissances l'est. 
\end{proof}

La notion de famille régulière de diviseurs s'étend aux champs de Deligne-Mumford localement noethériens, et le lemme \ref{famreglem} est encore valide. On a de plus le résultat utile suivant.

\begin{prop}
\label{parrec}
  Soit $X$ un champ de Deligne-Mumford localement noethérien, $\bold D$ une famille régulière de diviseurs. Alors le faisceau parabolique $\mathcal E_\cdot\in \obj \PAR_{\frac{1}{\bold r}}(X,\bold D)$ est localement libre si et seulement si :
  \begin{enumerate}
  \item[(i)] $\forall \bold l \in \frac{1}{\bold r}\mathbb Z$ le faisceau $\mathcal E_{\bold l}$ est localement libre sur $X$.
 \item[(ii)] $\forall i\in I$ $\forall l_i<l'_i\leq l_i+1 \in \frac{1}{r_i}\mathbb Z$ $\coker((\mathcal E_{l'_i})_\cdot\rightarrow(\mathcal E_{l_i})_\cdot)$ est localement libre vu comme objet de $ \PAR_{(\frac{1}{r_j})_{j\neq i}}(D_i,(D_j\cap D_i)_{j\neq i})$.
  \end{enumerate}
\end{prop}

\begin{proof}

Pour démontrer l'équivalence, on peut supposer que {\it (i)} est vrai.
Soit $J\neq \emptyset$, et $i\in J$. 
La donnée d'une facette $F:\{0<1\}^J\rightarrow  \frac{1}{\bold r}\mathbb Z^I$
équivaut à celle d'un triplet $(\tilde{F},l_i,l'_i)$, où
$\tilde{F}:\{0<1\}^{J-i}\rightarrow \prod_{j\in J-{i}} \frac{1}{r_j}\mathbb Z$
est une facette, et $l_i<l'_i\leq l_i+1$
dans  $\frac{1}{r_i}\mathbb Z$, tels que le diagramme suivant (bi)commute.

\xymatrix@R=12pt{
&&&&\{0<1\}^J \ar[r]^F& \frac{1}{\bold r}\mathbb Z^I\\
&&&&&\\
&&&& \{0<1\}^{J-i} \ar@<1ex>[uu]^0\ar@<-1ex>[uu]_1\ar[r]^{\tilde F}& \prod_{j\in J-{i}} \frac{1}{r_j}\mathbb Z\ar@<1ex>[uu]^{l_i}\ar@<-1ex>[uu]_{l'_i}
}
En notant $F^i_0$ et $F^i_1$ les deux foncteurs $\{0<1\}^{J-i}\rightarrow
\frac{1}{\bold r}\mathbb Z^I$ correspondants, on dispose d'un morphisme de complexes multiples 
$1>0:\mathcal E_\cdot\circ {F^i_1}^{op}\rightarrow \mathcal E_\cdot\circ
{F^i_0}^{op}$ et le fait qu'on suppose {\it (i)} vrai montre que ce morphisme
est injectif. Il en est donc de même pour le morphisme induit 
$\Tot(1>0) :\Tot(\mathcal E_\cdot\circ {F^i_1}^{op})\rightarrow \Tot(\mathcal
E_\cdot\circ {F^i_0}^{op})$, et on vérifie que 
$\Tot(\mathcal E_\cdot\circ F^{op})$ s'identifie canoniquement au cône
$\cone(\Tot(1>0))$ de celui-ci. On dispose donc d'un morphisme $\Tot(\mathcal
E_\cdot\circ F^{op})\rightarrow  \coker\Tot(1>0) $ 
qui est un quasi-isomorphisme (\cite{Weibel} 1.5.8), ce qui permet de conclure.

\end{proof}

\begin{rem}
\label{parrecrem}
  En particulier, le fait, pour un faisceau parabolique donné, d'être
  localement libre, ne dépend pas de l'ordre choisi sur l'ensemble d'indices
  $I$.   
\end{rem}

\subsubsection{Fibrés paraboliques relativement à une famille de diviseurs à croisements normaux simples}
\label{sec:fibr-parab-relat}

\begin{defi}
\label{famdivcroisnormsimpl}
  Une famille $\bold D=(D_i)_{i\in I}$ de diviseurs de Cartier effectifs sur un schéma localement noethérien $X$ est dite à croisements normaux simples si pour tout point $x$ de $\cup_{i\in I}D_i$ on a :

  \begin{enumerate}
  \item[(i)] l'anneau local $\mathcal O_{X,x}$ est régulier, 
  \item[(ii)] si $I_x=\{i\in I / x\in D_i\}$, et $s_i$ est une équation locale de $D_i$ en $x$, alors $\{s_i, i\in I_x\}$ est une partie d'un système régulier de paramètres.
\end{enumerate}
\end{defi}

\begin{rem}
  \begin{enumerate}
  \item C'est une légère adaptation de \cite{GM}, Definition 1.8.2.
  \item Il revient au même (\cite{GM}, Lemme 8.1.4) 
de dire que la famille est à croisements normaux et que chacun des $D_i$ est régulier.
  \item Comme un système régulier de paramètres est (\cite{EGA4.1}, Définition 17.1.6) une famille régulière, une famille de diviseurs à croisements normaux simples est en particulier une famille régulière au sens de la définition \ref{famregdef}. 
  \end{enumerate}
\end{rem}

Comme cette définition est invariante par un changement de base étale, elle a également un sens pour un champ de Deligne-Mumford. On a de plus :

\begin{prop}
\label{parfamdivcroisnormsimpl}
  Soit $X$ un champ de Deligne-Mumford localement noethérien, $\bold D$ une famille de diviseurs à croisements normaux simples. Alors le faisceau parabolique $\mathcal E_\cdot\in \obj \PAR_{\frac{1}{\bold r}}(X,\bold D)$ est localement libre si et seulement si $\forall \bold l \in \frac{1}{\bold r}\mathbb Z$ le faisceau $\mathcal E_{\bold l}$ est localement libre sur $X$.
\end{prop}

\begin{proof}
  Comme, pour tout $i\in I$, la famille $(D_j\cap D_i)_{j\neq i}$ est à
  croisements normaux simples sur $D_i$ (voir \cite{EGA4.1}, preuve de la
  Proposition 17.1.7), on peut raisonner par récurrence sur $\# I$. On conclut
  à l'aide de la proposition \ref{parrec} et du lemme suivant (je remercie Angelo Vistoli pour m'avoir fourni le principe de la preuve) :

  \begin{lem}
    Soit $R$ un anneau local noethérien régulier, d'idéal maximal $\mathfrak m$, $t\in \mathfrak m$, $t\notin \mathfrak m^2$. Soient de plus $M$ et $N$ deux modules libres de rang fini tel que $tM \subset N\subset M$. Alors $M/N$ est libre comme $R/t$-module.
  \end{lem}
  \begin{proof}
   De \cite{EGA4.1}, Corollaire 17.1.8, on déduit que l'anneau local $R/t$ est régulier, et  \cite{EGA4.1}, Proposition 16.3.7 montre qu'il est de dimension $\dim R-1$. La formule d'Auslander-Buchsbaum (\cite{EGA4.1}, Proposition 17.3.4) 
montre que le résultat à démontrer équivaut à $\prof_{R/t}M/N=\dim R-1$. Or
\cite{EGA4.1}, proposition 16.4.8 montre que $\prof_{R/t}M/N=\prof_R M/N$. De
plus, \cite{EGA4.1}, Corollaire 16.4.4 donne pour tout $R$-module de type fini
$P$ : si $k=R/\mathfrak m$ est le corps résiduel, $\prof_R P=\inf\{m\geq 0/\Ext^m_R(k,P)\neq 0\}$.
La suite exacte longue de cohomologie associée au foncteur $\Hom_R(k,\cdot)$ et à la suite exacte courte de $R$-modules 
$0\rightarrow N\rightarrow M\rightarrow M/N\rightarrow 0$, et une nouvelle application de la formule d'Auslander-Buchsbaum, permettent de conclure.

  \end{proof}
\end{proof}

\subsubsection{Revêtements de Kummer}
\label{sec:revet-kumm}

La définition des revêtements de Kummer adoptée ici est celle de \cite{GM}, \S 1 : donné un ensemble fini $I$, $\bold r=(r_i)_{i\in I}$ un ensemble d'entiers $r_i\geq 1$, des schémas localement noethériens $X$ et $Y$, $\bold s= (s_i)_{i\in I}$ un ensemble de sections régulières de $\mathcal O_X$, un morphisme $p:Y\rightarrow X$ est dit \emph{revêtement de Kummer} si $Y$ est muni d'une action du schémas en groupes $\boldsymbol{\mu}_{\bold r}=\prod_{i\in I}\boldsymbol{\mu}_{r_i}$ tel qu'il existe un $X$-isomorphisme $\boldsymbol{\mu}_{\bold r}$-équivariant de $Y$ avec $$\SPEC (\mathcal O_X[(t_i)_{i\in I}]/(t_i^{r_i}-s_i)_{i\in I})$$ 
Alors $p:Y\rightarrow X$ est l'application naturelle vers le quotient schématique.

Pour $i\in I$, on notera $D_i$ (resp. $E_i$) le diviseur de Cartier associé à $s_i$ (resp. $t_i$), et $\bold D=(D_i)_{i\in I}$ (resp. $\bold E=(E_i)_{i\in I}$) la famille correspondante.

\begin{lem}
\label{famregrev}
On suppose que $\bold D=(D_i)_{i\in I}$ est une famille régulière de diviseurs sur $X$. Alors la famille $\bold E=(E_i)_{i\in I}$ est une famille régulière de diviseurs sur $Y$.
\end{lem}

\begin{proof}
  Comme $p$ est plat, la famille  $p^*\bold D=(p^*D_i)_{i\in I}$ est régulière, par exemple d'après le lemme \ref{famreglem} $(ii)$. Or $(p^*D_i)_{i\in I}=(r_iD_i)_{i\in I}$, avec $r_i\geq 1$, et on peut appliquer le lemme \ref{famregprop} $(ii)$.
\end{proof}

\subsubsection{Fibrés paraboliques associés à un revêtement de Kummer}
\label{sec:fibr-parab-assoc}

 Soit $p:Y\rightarrow X$ un revêtement de Kummer. On conserve les notations de la partie \ref{sec:revet-kumm}. On appellera $\boldsymbol{\mu}_{\bold r}$-objet (ou  objet $\boldsymbol{\mu}_{\bold r}$-équivariant) d'un certain type sur $Y$ tout objet du même type sur le champ quotient $[Y|\boldsymbol{\mu}_{\bold r}]$. On notera 
$\boldsymbol{\mu}_{\bold r}\MOD Y$ la catégorie des faisceaux $\boldsymbol{\mu}_{\bold r}$-équivariants sur $Y$, $p_*^{\boldsymbol{\mu}_{\bold r}} : \boldsymbol{\mu}_{\bold r}\MOD Y \rightarrow \MOD X$ l'image directe le long de $[Y|\boldsymbol{\mu}_{\bold r}]\rightarrow X$, et $\boldsymbol{\mu}_{\bold r}\PAR_{\frac{1}{\bold r}}(Y,p^*\bold D)$ la catégorie des $\boldsymbol{\mu}_{\bold r}$-faisceaux paraboliques sur $Y$ le long de $p^*\bold D$ à poids multiples de $\frac{1}{\bold r}$ (ici $p^*\bold D=(p^*D_i)_{i\in I}$ est considérée comme une famille de diviseurs de Cartier effectifs $\boldsymbol{\mu}_{\bold r}$-équivariants sur $Y$ de la manière canonique).

Vu la platitude de $[Y|\boldsymbol{\mu}_{\bold r}]\rightarrow X$ et la
fonctorialité de $\PAR_{\frac{1}{\bold r}}(Y,p^*\bold D)$ en $X$ (\S
\ref{sec:oper-elem-sur}), on dispose d'un foncteur canonique 

$$p_*^{\boldsymbol{\mu}_{\bold r}}:\boldsymbol{\mu}_{\bold r}\PAR_{\frac{1}{\bold r}}(Y,p^*\bold D)
\rightarrow \PAR_{\frac{1}{\bold r}}(X,\bold D)$$

qu'on peut détailler ainsi : donné un objet 
$(\mathcal F_\cdot, k)$ de $\boldsymbol{\mu}_{\bold r}\PAR_{\frac{1}{\bold r}}(Y,p^*\bold D)$ ($k$ désignant l'isomorphisme des pseudo-périodes), 
on lui associe $(\mathcal E_\cdot, j)$, où $\mathcal E_\cdot=p_*^{\boldsymbol{\mu}_{\bold r}}\circ \mathcal F_\cdot$ (encore noté $p_*^{\boldsymbol{\mu}_{\bold r}}(\mathcal F_\cdot)$), et $j$ est donné par la $2$-composition :

\xymatrix@R=12pt{
  &&&&  (\mathbb Z ^I)^{op} \times (\frac{1}{\bold r}\mathbb Z^I)^{op} \ar[rr]^{+} \ar[dd]_{(\mathbb Z ^I)^{op} \times\mathcal F_\cdot} && (\frac{1}{\bold r}\mathbb Z^I)^{op}
   \ar[dd]^{\mathcal F_\cdot}\\
  &&&&&&\\
  &&&& (\mathbb Z ^I)^{op} \times \boldsymbol{\mu}_{\bold r}\MOD Y \ar[rr]_{\mathcal O_Y(-\cdot p^*\bold D)\otimes \cdot} \ar[dd]_{(\mathbb Z ^I)^{op} \times p_*^{\boldsymbol{\mu}_{\bold r}}} \ar@{=>}[uurr]^{k}&& \boldsymbol{\mu}_{\bold r}\MOD Y\ar[dd]^{p_*^{\boldsymbol{\mu}_{\bold r}}}\\
&&&&&&\\
  &&&& (\mathbb Z ^I)^{op} \times \MOD X \ar[rr]_{\mathcal O_X(-\cdot \bold D)\otimes \cdot} \ar@{=>}[uurr]^{proj}&& \MOD X}

la $2$-flèche $proj$ étant donnée par la formule de projection le long de $[Y|\boldsymbol{\mu}_{\bold r}]\rightarrow X$.

En composant ce foncteur avec le foncteur $\boldsymbol{\mu}_{\bold r}\MOD Y \rightarrow 
\boldsymbol{\mu}_{\bold r}\PAR_{\frac{1}{\bold r}}(Y,p^*\bold D)$ donné par 
$\mathcal F \rightarrow \mathcal F\otimes \mathcal O_Y(-\cdot\bold r\bold E)$
(voir \S \ref{sec:oper-elem-sur}) on obtient :

\begin{defi}
  On notera $\widehat{\ }_\cdot$ le foncteur 
$\boldsymbol{\mu}_{\bold r}\MOD Y\rightarrow \PAR_{\frac{1}{\bold r}}(X,\bold D)$ 
donné sur les objets par $\widehat{\mathcal F }_\cdot=p_*^{\boldsymbol{\mu}_{\bold r}}(\mathcal F\otimes \mathcal O_Y(-\cdot\bold r\bold E))$.

\end{defi}

\begin{prop}
\label{revKumfibpar} 
On suppose que $\bold D=(D_i)_{i\in I}$ est une famille régulière de diviseurs sur $X$, et que $\mathcal F$ est un $\boldsymbol{\mu}_{\bold r}$-faisceau localement libre de rang fini sur $Y$. Alors le faisceau parabolique sur $X$ associé $\widehat{\mathcal F }_\cdot$ est un fibré parabolique sur $X$ (au sens de la définition \ref{defpar}).
\end{prop}

\begin{proof}
  Soit $F:\{0<1\}^J\rightarrow \frac{1}{\bold r}\mathbb Z^I$ une facette 
correspondant à la donnée de $J$, $F(\bold 0)$, et de la famille
$\boldsymbol{\epsilon}$ (voir définition \ref{parpart}). Il s'agit de calculer l'homologie du complexe  
$$\Tot(\widehat{\mathcal F }_\cdot\circ F^{op})
\simeq p_*^{\boldsymbol{\mu}_{\bold r}} (\Tot(\mathcal F\otimes \mathcal O_X(-\cdot\bold r\bold E)\circ F^{op}))\simeq 
p_*^{\boldsymbol{\mu}_{\bold r}} (\mathcal F\otimes\Tot(\mathcal O_X(-\cdot\bold r\bold E)\circ F^{op}))$$

vu que le foncteur $\Tot$ commute à tout foncteur conservant les sommes
directes. Or il est clair d'après la définition \ref{parpart} que $F= F(\bold
0)+ F_{\frac{\boldsymbol{\epsilon}}{\bold r}J}$ et donc : 
$$\mathcal O_X(-\cdot\bold r\bold E)\circ F^{op}\simeq \mathcal O_Y(-\bold r F(\bold 0)\bold E)\otimes (\mathcal O_X(-\cdot\bold r\bold E)\circ F_{\frac{\boldsymbol{\epsilon}}{\bold r}J}^{op})$$

Le lemme \ref{Koszul2} montre qu'on doit calculer l'homologie du complexe

$$p_*^{\boldsymbol{\mu}_{\bold r}} (\mathcal F\otimes \mathcal O_Y(-\bold r F(\bold 0)\bold E)\otimes\Tot(\otimes_{i\in J}(\mathcal O_Y(-\epsilon_i E_i)\rightarrow \mathcal O_Y)))$$

Le foncteur $p_*^{\boldsymbol{\mu}_{\bold r}} (\mathcal F\otimes \mathcal
O_Y(-\bold r F(\bold 0)\bold E))\otimes\cdot)$ étant exact (en effet $p$ est
affine donc $p_*$ est exact [\cite{EGA3.1} 1.3.2], et ${\boldsymbol{\mu}_{\bold r}}$ est diagonalisable), on trouve donc

$$p_*^{\boldsymbol{\mu}_{\bold r}} (\mathcal F\otimes \mathcal O_Y(-\bold r F(\bold 0)\bold E)\otimes H_l(\Tot(\otimes_{i\in J}(\mathcal O_Y(-\epsilon_i E_i)\rightarrow \mathcal O_Y))))$$

Comme $\bold D=(D_i)_{i\in I}$ est une famille régulière de diviseurs sur $X$, il en est de même, d'après le lemme \ref{famregrev}, pour $\bold E=(E_i)_{i\in I}$. La définition \ref{facettedef} d'une facette impose que $\forall i\in J$, $\epsilon_i\geq 1$, et le lemme \ref{famregprop} montre que la famille $\boldsymbol{\epsilon}\bold E_{|J}=(\epsilon_i E_i)_{i\in J}$ est également une famille régulière de diviseurs.

On déduit du lemme \ref{famreglem} {\it(ii)} que $H_l(\Tot(\otimes_{i\in J}(\mathcal O_Y(-\epsilon_i E_i)\rightarrow \mathcal O_Y)))$ est nul pour $l>0$, et est isomorphe à $\mathcal O_{\cap_{i\in J} \epsilon_i E_i}$ pour $l=0$. On peut conclure à l'aide du lemme suivant :

\begin{lem}
 Pour tout $\boldsymbol{\mu}_{\bold r}$-faisceau $\mathcal F$ localement libre de rang fini sur $Y$, le faisceau 
$p_*^{\boldsymbol{\mu}_{\bold r}}(\mathcal F\otimes \mathcal O_{\cap_{i\in J} \epsilon_i E_i})$ est localement libre comme $\mathcal O_{\cap_{i\in J} D_i}$-module.
\end{lem}

\begin{proof}
  Vu la commutativité du diagramme 

\xymatrix@R=12pt{
&&&\cap_{i\in J}\epsilon_i E_i \ar[d]_q\ar[r]^{j_J} & Y\ar[d]^p\\
&&&\cap_{i\in J} D_i\ar[r]_{i_J} & X
}
on a $p_*^{\boldsymbol{\mu}_{\bold r}}(\mathcal F\otimes \mathcal O_{\cap_{i\in J} \epsilon_i E_i})\simeq p_*^{\boldsymbol{\mu}_{\bold r}}{j_J}_*{j_J}^*\mathcal F\simeq {i_J}_*q_*^{\boldsymbol{\mu}_{\bold r}}{j_J}^*\mathcal F$, et il s'agit donc de vérifier que $q_*^{\boldsymbol{\mu}_{\bold r}}{j_J}^*\mathcal F$ est localement libre. Comme $\boldsymbol{\mu}_{\bold r}$ est diagonalisable, c'est un facteur direct de 
$q_*{j_J}^*\mathcal F$, et vu que $\mathcal F$ localement libre de rang fini sur $Y$, il suffit de prouver que $q$ est fini et plat. Il est clairement fini comme composé de l'immersion fermée 
$\cap_{i\in J}\epsilon_i E_i\rightarrow p^*(\cap_{i\in J} D_i)$ (on rappelle que pour tout $i$ on a $\epsilon_i\leq r_i$) et de $p^*(\cap_{i\in J} D_i)\rightarrow \cap_{i\in J} D_i$ (fini car $p$ l'est). Pour la platitude on peut clairement supposer $X$ affine, soit $X=\spec R$. Mais alors 
$\cap_{i\in J}\epsilon_i E_i=\spec(\otimes_{i \in J} \frac{R}{s_i}\frac{[t_i]}{t_i^{\epsilon_i}}\otimes_R\otimes_{i \notin J}\frac{R[t_i]}{t_i^{r_i}-s_i})$. Or pour $i\in J$ (resp. $i\notin J$)  $t_i^{\epsilon_i}$ (resp. $t_i^{r_i}-s_i$) est unitaire, 
donc $\frac{R}{s_i}\frac{[t_i]}{t_i^{\epsilon_i}}$ (resp. $\frac{R[t_i]}{t_i^{r_i}-s_i}$) est plat sur $\frac{R}{s_i}$ (resp. sur $R$), donc $\otimes_{i \in J} \frac{R}{s_i}\frac{[t_i]}{t_i^{\epsilon_i}}\otimes_R\otimes_{i \notin J}\frac{R[t_i]}{t_i^{r_i}-s_i}$ est plat sur $\otimes_{i \in J} \frac{R}{s_i}$, qui est l'anneau définissant $\cap_{i\in J} D_i$.

\end{proof}

\end{proof}

\subsection{Fibrés paraboliques et champ des racines}
\label{sec:fibr-parab-champs}
 
Soit $S$ un schéma, et $X\rightarrow S$ un champ de Deligne-Mumford, qu'on supposera toujours localement noethérien.
 
\subsubsection{Champ des racines}

\label{sec:champ-des-racines}

Soit $r$ un entier supérieur ou égal à $1$, inversible dans $S$. 

\begin{defi}[\cite{AGV},\cite{Cadman}]
  \begin{enumerate}
  \item[(i)] Soit un couple $(\mathcal L,s)$ constitué d'un faisceau inversible sur $X$ et d'une section de ce faisceau. Soit $\mathcal U=[\mathbb A^1|\mathbb G_m]$ le champ classifiant les faisceaux inversibles muni d'une section. On appelle champ des racines $r$-ièmes de $(\mathcal L,s)$ le champ

$$\sqrt[r]{(\mathcal L,s)/X}=X\times_\mathcal U \mathcal U$$

où le produit fibré est pris par rapport aux morphismes  $(\mathcal L,s): X\rightarrow \mathcal U$, et l'élévation à la puissance $r$ : $\cdot^{\otimes r}:\mathcal U\rightarrow \mathcal U$. 
\item[(ii)] Soit $D$ un diviseur de Cartier effectif sur $X$, $s_D$ la section canonique de $\mathcal O_X(D)$. On note $\sqrt[r]{D/X}$ le champ 
$\sqrt[r]{(\mathcal O_X(D),s_D)/X}$.
\end{enumerate}
\end{defi}

Soit $I$ un ensemble fini, $\bold r=(r_i)_{i\in I}$ un ensemble d'entiers $r_i\geq 1$, inversibles dans $S$.

\begin{defi}

\begin{enumerate}
\item 
Soit $(\bold {\mathcal L},\bold s)=(\mathcal L_i,s_i)_{i\in I}$ un ensemble de faisceaux inversibles sur $X$ muni chacun d'une section. On note 
$\sqrt[\bold r]{(\bold {\mathcal L},\bold s)/X}$ le champ 
$\times_{i\in I}\sqrt[r_i]{(\mathcal L_i,s_i)/X}$.

 \item Soit
 $\bold D=(D_i)_{i\in I}$ un ensemble de diviseurs de Cartier effectifs sur $X$. On note  $\sqrt[\bold r]{\bold D/X}$ le champ $\times_{i\in I}\sqrt[r_i]{D_i/X}$.
\end{enumerate}

\end{defi}

\begin{prop}
\label{champracchampquot}
Soit $(\bold {\mathcal L},\bold s)=(\mathcal L_i,s_i)_{i\in I}$ un ensemble de faisceaux inversibles sur $X$ munis de sections.
  Supposons qu'on dispose sur $X$ de faisceaux inversibles $\mathcal N_i$ et 
d'isomorphismes $\psi_i:\mathcal N_i^{\otimes r_i}\simeq  \mathcal L_i$.
Il existe alors un isomorphisme naturel de champs sur $X$ :

$$\sqrt[\bold r]{(\bold {\mathcal L},\bold s)/X}
\simeq [ \SPEC (\frac{\Sym(\oplus_{i\in I}\mathcal N_i^{\vee})}{(\mathcal N_i^{\otimes r_i}\simeq  \mathcal L_i)_{i\in I}})|\boldsymbol{\mu}_{\bold r} ]$$
\end{prop}

\begin{proof}
  Le cas où $\#I =1$ est traité dans \cite{Cadman}, version 1, Proposition 3.2, (voir aussi \cite{Borne}, théorème 4), et le cas général en résulte immédiatement. 
\end{proof}

\begin{cor}
\label{champracchampquot2}
 Soit  $\bold s= (s_i)_{i\in I}$ un ensemble de sections régulières de $\mathcal O_X$, $D_i=(s_i)$ les diviseurs de Cartier correspondants. Alors il existe un isomorphisme naturel de champs sur $X$ :

$$  \sqrt[\bold r]{\bold D/X}     \simeq [\SPEC (\mathcal O_X[(t_i)_{i\in I}]/(t_i^{r_i}-s_i)_{i\in I}) |\boldsymbol{\mu}_{\bold r} ]$$                
\end{cor}

\begin{proof}
 Découle directement de la proposition \ref{champracchampquot}.
\end{proof}
\subsubsection{La correspondance : énoncé}
\label{sec:la-correspondance-enonce}

Soit $(\bold {\mathcal L},\bold s)=(\mathcal L_i,s_i)_{i\in I}$ un ensemble de faisceaux inversibles sur $X$ munis de sections.
Sur $\sqrt[\bold r]{(\bold {\mathcal L},\bold s)/X}$, on dispose d'une racine $\bold r$-ième canonique $(\bold {\mathcal N},\bold t)=(\mathcal N_i,t_i)_{i\in I}$ de $(\bold {\mathcal L},\bold s)$.

On note $\pi:\sqrt[\bold r]{(\bold {\mathcal L},\bold s)/X}\rightarrow X$ le morphisme canonique.

Dans le cas particulier où les couples $(\mathcal L_i,s_i)$ sont associés à des diviseurs effectifs $D_i$ sur $X$ (i.e. $(\mathcal L_i,s_i)=(\mathcal O_X(D_i),s_{D_i})$ pour tout $i\in I$), les relations  $\pi^*s_i=t_i^{\otimes r_i}$, et le fait que $\pi$ est plat, montrent que les $t_i:\mathcal O_{ \sqrt[\bold r]{\bold D/X} }\rightarrow \mathcal N_i$ sont des monomorphismes (i.e. les $t_i$ sont des sections régulières), si bien que les couples $(\mathcal N_i,t_i)$ sont associés à des diviseurs de Cartier effectifs $E_i$ sur $\sqrt[\bold r]{\bold D/X}$.

En imitant la construction du \S \ref{sec:fibr-parab-assoc} on obtient un foncteur :

\xymatrix@R=2pt{
&\widehat{\ }_\cdot:&\MOD(\sqrt[\bold r]{\bold D/X})\ar[r]& \PAR_{\frac{1}{\bold r}}(X,\bold D)\\
&&\mathcal F \ar[r]& \widehat{\mathcal F}_\cdot=\pi_*(\mathcal F\otimes  \mathcal O_{ \sqrt[\bold r]{\bold D/X} }(-\cdot\bold r\bold E))
}

\begin{thm}
  \label{fibparfibchamp}
On suppose que $\bold D$ est une famille régulière de diviseurs sur le champ de Deligne-Mumford $X$. Alors le foncteur $\widehat{\ }_\cdot$ induit une équivalence de catégories tensorielles entre $\Vect(\sqrt[\bold r]{\bold D/X})$ et $\Par_{\frac{1}{\bold r}}(X,\bold D)$.
\end{thm}

\subsubsection{Bonne définition}
\label{sec:bonne-definition}

Il s'agit de voir que si $\mathcal F\in\obj \Vect(\sqrt[\bold r]{\bold D/X})$, alors le faisceau parabolique $\widehat{\mathcal F}_\cdot$ est localement libre. Comme il s'agit d'une question locale pour la topologie étale sur $X$, on peut, quitte à prendre un atlas étale, supposer que $X$ est un schéma. Quitte à localiser encore de façon à trivialiser chacun des diviseurs de la famille $\bold D$, le corollaire \ref{champracchampquot2} montre qu'on peut supposer que 
$\sqrt[\bold r]{\bold D/X}\rightarrow X$ est du type $[Y|\boldsymbol{\mu}_{\bold r}]\rightarrow X$, où $Y\rightarrow X$ est un revêtement de Kummer ramifié le long de $\bold D$. Mais alors la proposition \ref{revKumfibpar} permet de conclure.

\subsubsection{Équivalence réciproque}
\label{sec:equiv-recipr}

Soit $\mathcal E_\cdot \in\obj\Par_{\frac{1}{\bold r}}(X,\bold D)$. On pose 

$$\widehat{\mathcal E_\cdot}=\int^{\frac{1}{\bold r}\mathbb Z}\pi^*\mathcal E_\cdot \otimes\mathcal N^{\otimes\bold r\cdot}$$

où $\int^{\frac{1}{\bold r}\mathbb Z}$ désigne la cofin (coend), voir \cite{MacLane}.
C'est à priori un élément de $\MOD (\sqrt[\bold r]{\bold D/X})$, mais on va voir que c'est en fait un faisceau localement libre.

\begin{lem}
\label{adjonction}
On fixe $i\in I$.
  Soient les foncteurs

\xymatrix@R=2pt{
\PAR_{\frac{1}{\bold r}}(X,\bold D) &\PAR_{(\frac{1}{r_j})_{j\neq i}}(
\sqrt[r_i]{D_i/X},(\pi_i^*D_j)_{j\neq i})\\
L_i:\mathcal E_\cdot \ar[r] &L_i\mathcal E_\cdot=\int^{l_i\in \frac{1}{r_i}\mathbb Z} \pi_i^*(\mathcal E_{l_i})_\cdot \otimes\mathcal N_i^{\otimes r_il_i}\\
R_i\mathcal F_\cdot :(l_i\rightarrow {\pi_i}_*(\mathcal F_\cdot\otimes \mathcal N_i^{\otimes{-l_ir_i}}))& \mathcal F_\cdot : R_i \ar[l] }

où $\pi_i :\sqrt[r_i]{D_i/X}\rightarrow X$ désigne la projection canonique, $\mathcal N_i$ la racine $r_i$-ème canonique de $\mathcal O_X(D_i)$ sur 
$\sqrt[r_i]{D_i/X}$, et $(\mathcal E_{l_i})_\cdot$ la restriction de $\mathcal E_\cdot$ via le foncteur $\prod_{j\neq i}\frac{1}{r_j}\mathbb Z\rightarrow \frac{1}{\bold r}\mathbb Z$ induit par $l_i$.

Ces foncteurs sont adjoints, $L_i$ étant adjoint à gauche et $R_i$ adjoint à droite.
\end{lem}

\begin{proof}
  Cela résulte de la définition des cofins.
\end{proof}

\begin{lem}
  Le foncteur $L_i$ envoie $\Par_{\frac{1}{\bold r}}(X,\bold D)$ sur $\Par_{(\frac{1}{r_j})_{j\neq i}}(\sqrt[r_i]{D_i/X},(\pi_i^*D_j)_{j\neq i})$.
\end{lem}

\begin{proof}
  On procède par récurrence sur $\# I$.
On commence par le cas où $\# I=1$. Le cas où $X$ est un schéma est traité dans \cite{Borne}. Le cas général s'y ramène grâce au lemme suivant :

\begin{lem}
Soit $X$ un champ de Deligne-Mumford, $\mathcal E : K \rightarrow \Vect X$ un diagramme, $p:X_0\rightarrow X$ un atlas étale.
Si $\lim_{\stackrel{\rightarrow}{k\in K}} p^*\mathcal E_k$ existe dans
$\Vect X_0$, alors $\lim_{\stackrel{\rightarrow}{k\in K}} {\mathcal E_k}$
  existe dans $\Vect X$.
\end{lem}

\begin{proof}
C'est à peu près direct à partir de la description suivante des objets de  
$\Vect X$ : soit $X_1=X_0\times_X X_0$, et $s,b:X_1\rightrightarrows X_0$ le groupoïde correspondant. La catégorie $\Vect X$ est équivalente à la catégorie des couples $(\mathcal E_0,\alpha)$, où $\mathcal E_0\in \obj \Vect X_0$, et $\alpha :s^* \mathcal E_0\simeq b^*\mathcal E_0$ est une donnée de descente, i.e. vérifie la condition de descente : $m^*\alpha = pr_1^*\alpha\circ pr_2^* \alpha$, où $pr_1, pr_2,  m : X_1\times_{X_0}X_1\rightarrow X_1$ désignent respectivement, les projections et la multiplication dans le groupoïde.

\end{proof}

On revient au cas général ($\# I$ quelconque) : 
pour montrer que $\int^{l_i\in \frac{1}{r_i}\mathbb Z}
\pi_i^*(\mathcal E_{l_i})_\cdot \otimes\mathcal N_i^{\otimes r_il_i}$ est un
faisceau parabolique localement libre, le plus simple est d'appliquer la proposition \ref{parrec}. La partie $(i)$ du critère résulte du cas $\# I=1$. Pour vérifier la partie $(ii)$ de ce critère, on fixe $j\neq i$, et $l'_j<l_j\leq l_j+1 \in \frac{1}{r_j}\mathbb Z$. Il s'agit de voir que

$$\coker(\int^{l_i\in \frac{1}{r_i}\mathbb Z}
\pi_i^*(\mathcal E_{l_il_j'})_\cdot \otimes\mathcal N_i^{\otimes r_il_i}
\rightarrow \int^{l_i\in \frac{1}{r_i}\mathbb Z}
\pi_i^*(\mathcal E_{l_il_j})_\cdot \otimes\mathcal N_i^{\otimes r_il_i})$$ 

est un objet de $\Par_{(\frac{1}{r_k})_{k\neq i,j}}(
\pi_i^*D_j,(\pi_i^*(D_k\cap D_j))_{k\neq i,j})$. Mais comme le foncteur 
$L_i$ est adjoint à gauche (lemme \ref{adjonction}), 
donc exact à droite, et $\pi_i^*D_j=\sqrt[r_i]{D_j\cap D_i/D_j}$, cela résulte de l'hypothèse de récurrence. 

\end{proof}

\subsubsection{Preuve de l'équivalence}
\label{sec:preuve-de-leq}

\begin{lem}
  Le foncteur $R_i$ envoie $\Par_{(\frac{1}{r_j})_{j\neq i}}(\sqrt[r_i]{D_i/X},(\pi_i^*D_j)_{j\neq i})$ sur $\Par_{\frac{1}{\bold r}}(X,\bold D)$, et est une équivalence réciproque de la restriction de $L_i$ à $\Par_{\frac{1}{\bold r}}(X,\bold D)$.
\end{lem}

\begin{proof}

On commence par remarquer que si la première assertion est vérifiée, la seconde à un sens, et est vraie : en effet on peut se contenter de vérifier que $L_i$ et $R_i$ sont des isomorphismes après évaluation des variables, et on se ramène donc au cas où $\#I=1$. C'est de nouveau un problème local pour la topologie étale, et on peut donc se ramener au cas où $X$ est un schéma. Pour ce dernier cas, on renvoie à \cite{Borne}.

On montre l'ensemble des deux assertions par récurrence sur $\#I$. 
Pour le cas où $\#I=1$ : la première assertion résulte de la partie \ref{sec:bonne-definition}, la seconde s'ensuit.
Pour $I$ quelconque on applique l'hypothèse de récurrence qui permet de dire que $\Vect \sqrt[\bold r]{\bold D/X} \simeq \Par_{(\frac{1}{r_j})_{j\neq i}}(\sqrt[r_i]{D_i/X},(\pi_i^*D_j)_{j\neq i})$, et à nouveau la partie \ref{sec:bonne-definition} permet de conclure à la validité de la première assertion, et donc de la seconde.
\end{proof}

\subsubsection{Preuve du caractère tensoriel}
\label{sec:preuve-du-car}
Vu l'expression du produit tensoriel donnée \S \ref{sec:oper-elem-sur}, le fait que l'équivalence soit compatible au produit tensoriel résulte de la formule de Fubini pour les cofins (voir \cite{MacLane}, et \cite{Borne} pour le détail dans le cas où $\# I=1$).

\subsubsection{Structure locale des fibrés paraboliques}

\begin{cor}
	\label{strulocfibpar}
	Avec les notations du théorème \ref{fibparfibchamp}, pour tout fibré parabolique $\mathcal E_\cdot\in \obj(\Par_{\frac{1}{\bold r}}(X,\bold D))$, et tout point $x\in X$, il existe un voisinage étale $U\rightarrow X$ de $x$ tel que ${\mathcal E_\cdot}_{|U}$ soit une somme directe finie de fibrés paraboliques inversibles. 
\end{cor}

\begin{proof}
	Vu le théorème \ref{fibparfibchamp}, il suffit de montrer la propriété correspondante pour les fibrés champêtres, mais alors la preuve de \cite{Borne}, Proposition 3.2, s'adapte immédiatement.
\end{proof}

\begin{rem}
	\label{locab}
\begin{enumerate}
	\item Si $X$ est de plus un schéma, on peut imposer à $U\rightarrow X$ d'être un ouvert pour la topologie de Zariski.
	\item La preuve montre qu'on peut choisir les fibrés paraboliques de la forme $\underline{\mathcal O_X}_\cdot[\bold l]$, pour $\bold l\in \obj \frac{1}{\bold r}\mathbb Z^I$. En termes champêtres, si $X=\spec R$, où $R$ est un anneau local, alors $\Pic (\sqrt[\bold r]{\bold D|X})\simeq \frac{\frac{1}{\bold r}\mathbb Z^I}{\mathbb Z^I}$, voir \S \ref{grpicchamprac}.
	\item 		On suppose de plus que $\bold D$ une famille de diviseurs à croisements normaux simples. Alors la proposition \ref{parfamdivcroisnormsimpl} et le corollaire \ref{strulocfibpar} montrent que si les composantes du faisceau parabolique $\mathcal E_\cdot\in \obj \PAR_{\frac{1}{\bold r}}(X,\bold D)$ sont localement libres, il est localement abélien au sens de \cite{IyerSimpson}, Definition 2.2.

	\end{enumerate}
\end{rem}

\subsubsection{Groupe de Picard des champs des racines}
\label{grpicchamprac}
Le corollaire suivant est énoncé, dans le cas particulier des courbes tordues, dans \cite{Brochard}. Le théorème \ref{fibparfibchamp} n'est pas indispensable pour le démontrer (on peut aussi utiliser \cite{Cadman}, Corollary 3.1.2), mais en fournit une preuve commode.

\begin{cor}
	\label{picchamprac}
	Avec les notations du théorème \ref{fibparfibchamp}, on a une suite exacte naturelle
\begin{displaymath}
	0 \rightarrow \Pic X \rightarrow \Pic (\sqrt[\bold r]{\bold D|X})\rightarrow \prod_{i\in I} H^0(D_i, \frac{\mathbb Z}{r_i})\rightarrow 0
\end{displaymath}
\end{cor}

\begin{proof}
On peut supposer les $D_i$ connexes (en effet si $D$ et $D'$ sont deux diviseurs de Cartier effectifs à supports disjoints, et $r\geq 1$ est un entier, alors 
$\sqrt[r]{D+D'|X}\simeq \sqrt[r]{D|X}\times_X\sqrt[r]{D'|X}$).

On note $\mathcal N_i$ la racine $r_i$-ème canonique de $\mathcal O_X(D_i)$ sur $\sqrt[\bold r]{\bold D|X}$, et $\pi:\sqrt[\bold r]{\bold D|X}\rightarrow X$ le morphisme naturel. On va montrer plus précisément : il existe un unique morphisme surjectif $\Pic (\sqrt[\bold r]{\bold D|X})\rightarrow \prod_{i\in I}\frac{\mathbb Z}{r_i}$ envoyant $[\mathcal N_i]$ sur le générateur canonique de la $i$-ème composante, et dont le noyau est $\pi^*: \Pic X \rightarrow \Pic (\sqrt[\bold r]{\bold D|X})$.

Il suffit de le montrer pour $\# I =1$. En effet en supposant cette propriété vérifiée dans ce cas, le morphisme naturel

\begin{displaymath}
	\frac{\Pic (\sqrt[\bold r]{\bold D|X})}{\Pic X}\rightarrow 
	\prod_{i\in I}	\frac{\Pic \left(\sqrt[\bold r]{\bold D|X}\right)}{\Pic\left(\sqrt[(r_j)_{j\neq i}]{(D_j)_{j\neq i}|X}\right)}
\end{displaymath}

envoie  $[\mathcal N_i]$ sur le générateur canonique de la $i$-ème composante, et est donc surjectif. Comme une récurrence immédiate donne l'égalité des cardinaux, c'est un isomorphisme.

Reste à voir le cas où $\#I=1$. Cela résulte immédiatement de l'assertion : pour tout faisceau inversible $\mathcal K$ sur $\sqrt[r]{D|X}$, il existe un unique entier $l$ dans $\{ 0,\cdots, r-1\}$ tel qu'il existe $\mathcal M$ faisceau inversible sur $X$ tel que $\mathcal K\otimes \mathcal N^{\otimes l}\simeq \pi^* \mathcal M$. En traduisant en termes de fibrés paraboliques grâce au théorème \ref{fibparfibchamp}\footnote{en fait la version à indice unique montrée dans \cite{Borne}.} on voit qu'il faut montrer : pour tout faisceau parabolique inversible $\mathcal K_\cdot$ à poids dans $\frac{1}{r}\mathbb Z$, il existe un unique entier $l$ dans $\{ 0,\cdots, r-1\}$ tel qu'il existe $\mathcal M$ faisceau inversible sur $X$ tel que  $\mathcal K_\cdot[\frac{l}{r}]\simeq \underline{\mathcal M}_\cdot$ (le faisceau parabolique à structure spéciale associé à $\mathcal M$, voir \S \ref{sec:oper-elem-sur}). La donnée de $\mathcal K_\cdot$ équivaut à celle d'une filtration

\begin{displaymath}
	\mathcal K_0\supset\mathcal K_{\frac{1}{r}}\supset \cdots \supset \mathcal K_{1-\frac{1}{r}}\supset \mathcal K_1\simeq \mathcal K_0\otimes_{\mathcal O_X}\mathcal O_X(-D)
\end{displaymath}
telle que pour $l\leq l'$, le faisceau $\mathcal K_{\frac{l}{r}}/\mathcal K_{\frac{l'}{r}}$ est localement libre sur $\mathcal O_D$. Ceci implique l'égalité des rangs :

\begin{displaymath}
	\rg({\mathcal K_0}_{|D})=  \sum_{l=0}^{r-1}\rg\left(\frac{\mathcal K_{\frac{l}{r}}}{\mathcal K_{\frac{l+1}{r}}}\right)\end{displaymath}

	mais comme $\mathcal K_0$ est inversible et $D$ est connexe, il existe un unique entier $l$ dans $\{ 0,\cdots, r-1\}$ tel que $\mathcal K_{\frac{l}{r}}\neq \mathcal K_{\frac{l+1}{r}}$, et pour cet entier  $\mathcal K_\cdot[\frac{l}{r}]\simeq \underline{\mathcal K_0}_\cdot$.

\end{proof}

\subsubsection{Image directe de fibrés paraboliques}

\label{imdirfibpar}
On se donne $S$ un schéma de base, $X\rightarrow S$ (respectivement $Y\rightarrow S$) un $S$-champ de Deligne-Mumford localement noethérien, $\bold D=(D_i)_{i\in I}$ (respectivement $\bold E=(E_j)_{j\in J}$) une famille régulière de diviseurs sur $X$ (respectivement sur $Y$), et $\bold r=(r_i)_{i\in I}$ (respectivement $\bold s=(s_j)_{j\in J}$) une famille d'entiers (supérieurs à $1$, inversibles dans $S$).
On fixe de plus $p:Y\rightarrow X$ un $S$-morphisme représentable fini et plat, $\alpha :J\rightarrow I$ une application, vérifiant 

\begin{enumerate}
	\item $\forall j\in J \;\; s_j|r_{\alpha(j)}$
	\item $\forall i\in I \;\; p^*D_i=\sum_{j\in\alpha^{-1}(i)}\frac{r_i}{s_j}E_j$
\end{enumerate}

Enfin, on note $q: \sqrt[\bold s]{\bold E/Y} \rightarrow \sqrt[\bold r]{\bold D/X}$ le morphisme naturel.

\begin{defi}
	Sous les conditions ci-dessus, on définit l'image directe d'un fibré parabolique par la formule :

	\xymatrix@R=2pt{
&p_*:& \Par_{\frac{1}{\bold s}}(Y,\bold E) \ar[r]& \Par_{\frac{1}{\bold r}}(X,\bold D)\\
&&\mathcal E_\cdot \ar[r]& (p_*{\mathcal E_\cdot})_\cdot=(\frac{\bold l}{\bold r}\rightarrow p_*(\mathcal E_{\frac{\bold l\circ\alpha}{\bold s}}))
}

\end{defi}

\begin{prop}
Les foncteurs d'images directes champêtre et parabolique sont compatibles : on a un isomorphisme fonctoriel en $\mathcal E_\cdot\in \obj \Par_{\frac{1}{\bold s}}(Y,\bold E)$ : $\widehat{(p_*{\mathcal E_\cdot})_\cdot}\simeq q_*(\widehat{\mathcal E_\cdot})$.
\end{prop}

\begin{proof}
	On note $\pi : \sqrt[\bold r]{\bold D/X} \rightarrow X$ et $\varpi : \sqrt[\bold s]{\bold E/Y} \rightarrow Y$ les morphismes canoniques. On a un $2$-isomorphisme naturel $\pi\circ q\simeq p\circ \varpi$.

	Pour tout $i$ dans $I$ (respectivement $j$ dans $J$), soit $\mathcal M_i$ (respectivement $\mathcal N_j)$ la racine $r_i$-ème (respectivement $s_j$-ème) canonique de $\mathcal O_X(D_i)$ (respectivement de $\mathcal O_Y(E_j)$) sur $\sqrt[\bold r]{\bold D/X}$ (respectivement sur $\sqrt[\bold s]{\bold E/Y})$, elle est muni de sa section canonique. Le morphisme $q$ est défini par la condition : pour tout $i$ dans $I$, $q^*(\mathcal M_i)=\otimes_{j\in \alpha^{-1}(i)}\mathcal N_j$, et la condition correspondante évidente sur les sections.

	On fixe $\mathcal E_\cdot\in \obj \Par_{\frac{1}{\bold s}}(Y,\bold E)$ et $(\frac{\bold l}{\bold r})=(\frac{l_i}{r_i})_{i\in I} \in \obj(\frac{1}{\bold r}\mathbb Z)$, et on calcule :
	\begin{align*}
		\pi_*(q_*(\widehat{\mathcal E_\cdot})\otimes_{i\in I}\mathcal M_i^{\otimes -l_i}) & \simeq 
		\pi_*\left(q_*\left(\left(\int^{\frac{1}{\bold s}\mathbb Z}\varpi^*\mathcal E_\cdot \otimes\mathcal N^{\otimes\bold s\cdot}\right)\otimes_{i\in I}\otimes_{j\in \alpha^{-1}(i)}\mathcal N_j^{\otimes -l_i}\right)\right)\\
		& \simeq p_*\left(\varpi_*\left(\int^{\frac{1}{\bold s}\mathbb Z}\varpi^*\mathcal E_\cdot \otimes\mathcal N^{\otimes\bold s\cdot -\bold l\circ\alpha}\right)\right)\\
		& \simeq p_*\left(\varpi_*\left(\int^{\frac{1}{\bold s}\mathbb Z}\varpi^*\mathcal E_\cdot[\frac{\bold l\circ\alpha]}{\bold s}] \otimes\mathcal N^{\otimes\bold s\cdot}\right)\right)\\
		& \simeq p_*(\mathcal E_{\frac{\bold l\circ\alpha}{\bold s}})
\end{align*}

d'où la conclusion.

\end{proof}

\section{Groupe fondamental modéré comme groupe fondamental champêtre}
\label{sec:groupe-fond-modere}

\subsection{Groupe fondamental champêtre}
\label{sec:groupe-fond-champ}

Noohi (\cite{Noohi}) et Zoonekynd (\cite{Zoon}) ont étendu la théorie classique du groupe fondamental profini de \cite{SGA1} du cas d'un schéma à celui d'un champ de Deligne-Mumford. On rappelle brièvement leur définition. 

\begin{defi}[\cite{SGA1},\cite{Noohi},\cite{Zoon}]
   Soit $X$ un champ de Deligne-Mumford. On note $\Rev X$ la $2$-sous-catégorie pleine de la $2$-catégorie $\Champs X$ des champs sur $X$ dont les objets sont les morphismes $Y\rightarrow X$ représentables étales finis. On note $\Cat\Rev X$ la catégorie associée (i.e. la catégorie dont les morphismes sont les classes de $2$-isomorphisme de $1$-morphismes  dans $\Rev X$).
\end{defi}

\begin{thm}[\cite{Noohi} Theorem 4.2, \cite{Zoon} \S 3]
  Si $X$ est un champ de Deligne-Mumford connexe, et $x:\spec \Omega \rightarrow X$ un point géométrique, la paire $(\Cat\Rev X, x^*)$ est une catégorie galoisienne au sens de \cite{SGA1}. \end{thm}

\begin{defi}
\label{defgrfundchamp}
  Avec les notations du théorème, on notera $\pi_1(X,x)$ le groupe fondamental de la catégorie galoisienne  $(\Cat\Rev X, x^*)$.
\end{defi}

\subsection{Groupe fondamental modéré}
\label{sec:groupe-fond-modere-1}

On rappelle le résultat suivant :

\begin{thm}[\cite{GM} Theorem 2.4.2]
\label{grfundmod}
Soit $X$ un schéma localement noethérien, normal, connexe, $D$ un diviseur à croisements normaux, et $x:\spec \Omega\rightarrow X$ un point géométrique, $x\notin D$.
 La catégorie $\Rev^D (X)$ des revêtements de $X$ modérément ramifiés le long de $X$ est une catégorie galoisienne, dont on note $\pi_1^D(X,x)$ le groupe fondamental. \end{thm}

On va se restreindre à étudier le cas d'une famille de diviseurs à croisements normaux simples (définition \ref{famdivcroisnormsimpl}). Le but de ce paragraphe est de démontrer :

\begin{prop}
\label{grfundmodgrfundchamp}
Avec les hypothèses du théorème \ref{grfundmod}, si on suppose de plus $X$ défini sur un corps $k$, et $D$ est la réunion d'une famille $\bold D=(D_i)_{i\in I}$ de diviseurs irréductibles à croisements normaux simples, alors il existe un isomorphisme naturel :

$$\pi_1^D(X,x)\simeq \varprojlim_{\bold r}\pi_1(\sqrt[\bold r]{\bold D/X},x)$$

où la limite est prise sur les multi-indices $\bold r=(r_i)_{i\in I}$ d'entiers non divisibles par la caractéristique $p$ de $k$.
\end{prop}

Dans le reste de ce paragraphe \ref{sec:groupe-fond-modere}, on conservera les hypothèses de la proposition \ref{grfundmodgrfundchamp}. On va montrer que l'on a une équivalence naturelle de catégories galoisiennes\footnote{la notion de $2$-limite filtrée de catégories utilisée ici et par la suite est précisée dans l'appendice \ref{2limfil}}:

$$\Rev^D (X) \simeq \varinjlim_{\bold r}\Cat \Rev (\sqrt[\bold r]{\bold D/X})$$ 

La compatibilité de cet isomorphisme aux foncteurs fibres induits par $x$ impliquera bien la proposition \ref{grfundmodgrfundchamp}\footnote{Comme $\sqrt[\bold r]{\bold D/X}$ admet $X$ pour espace des modules, $x$ définit aussi un point géométrique de ce champ}.

\subsection{Le foncteur $C$}
\label{sec:le-foncteur-c}

\begin{lem}
\label{abh}
Soit $\pi:Y\rightarrow X$ dans $\obj \Rev^D (X)$ galoisien de groupe $G$ (au sens de \cite{GM} 2.4.5), de multi-indice de ramification $\bold r$. Alors le morphisme naturel de champs $[Y|G]\rightarrow \sqrt[\bold r]{\bold D/X}$ est un isomorphisme.
\end{lem}

\begin{proof}
  Cela résulte, essentiellement, de l'hypothèse que $D$ est à croisements normaux simples, et du lemme d'Abhyankar. 

En détail : on précise d'abord la définition du morphisme naturel $[Y|G]\rightarrow \sqrt[\bold r]{\bold D/X}$. Soit $\bold E=(E_i)_{i\in I}$ la famille de diviseurs sur $Y$ défini par $\forall i\in I\; E_i=(\pi^*D_i)_{red}$, si bien que $\forall i\in I\; \pi^*D_i=r_i E_i$. 

Soit $S$ un schéma, et $(f,p)$ un objet de 
$[Y|G](S)$, i.e. une paire constituée d'un $G$-torseur $p:T\rightarrow S$ et d'un morphisme $G$-équivariant $f: T\rightarrow Y$. On a les quotients schématiques $S=T/G$ et $X=Y/G$, si bien qu'il existe un unique morphisme $g: S\rightarrow X$ tel que $g\circ p=\pi \circ f$. Pour tout $i\in I$, le faisceau $\mathcal O_T(f^* E_i)$ définit une racine $r_i$-ème de $D_i$ sur $T$, comme ce faisceau est $G$-équivariant, on peut le descendre canoniquement le long du $G$-torseur  $p:T\rightarrow S$, et les racines $r_i$-ièmes $p_*^G(\mathcal O_T(f^* E_i))$ de $D_i$ sur $S$ définissent un objet de $\sqrt[\bold r]{\bold D/X}(S)$.

Pour montrer que ce morphisme $[Y|G]\rightarrow \sqrt[\bold r]{\bold D/X}$ est
un isomorphisme, il suffit de le vérifier sur les fibres géométriques, et on
peut donc supposer que $X=\spec R$, où $R$ est un anneau local noethérien strictement hensélien. Comme c'est évident en dehors du support de $\bold D$,
on peut supposer de plus $R$ régulier.

On choisit pour tout $i\in I$ une équation locale $s_i$ de $D_i$. 
On pose $R'=\frac{R[(T_i)_{i\in I}]}{(T_i^{r_i}-s_i)_{i\in I}}$, où les
$(T_i)_{i\in I}$ sont des indéterminées, et $X'=\spec R'$.  Le morphisme  $X'\rightarrow X$ est modérément ramifié (\cite{GM} Example 2.2.4). Comme les $s_i$ sont tous dans l'idéal maximal $\mathfrak m$ de $R$, $R'$ est un anneau local. Vu que la famille $(D_i)_{i\in I}$ est à croisements normaux simples, $R'$ est même régulier (\cite{GM}, Proposition 1.8.5).

On a encore $Y=\spec S$, où $S$ est une $R$-algèbre finie, donc un produit fini d'anneaux locaux strictement henséliens (en effet l'hypothèse de modération impose que les extensions résiduelles sont séparables [\cite{GM} Definition 2.1.2], donc ici triviales). On réduit facilement le problème au cas où $Y$ est irréductible. Vu que les $r_i$ sont non nuls résiduellement, on peut choisir des équations $t_i$ des $E_i$ telles que $t_i^{r_i}=s_i$ dans $S$. Celles-ci définissent un $X$-morphisme  $Y\rightarrow X'$.

 Soit $Y'$ le produit fibré de $Y$ et $X$ sur $X'$ dans $\Rev^D (X)$, c'est à dire la normalisation du produit fibré schématique $Y\times_{X}X'$. On a donc $Y=\spec S'$, où $S'$ est la clôture intégrale de $R'$ dans l'extension des anneaux  de fractions totaux $R(Y)/R(X)$. Comme celle-ci est finie étale d'après l'hypothèse de modération, le morphisme $Y'\rightarrow X'$ est fini. Comme $Y'\rightarrow X$ est modéré, toute composante irréductible de $Y'$ domine $X$, et donc toute composante irréductible de $Y'$ domine $X'$. On peut donc appliquer le théorème de pureté de Zariski-Nagata (\cite{SGA1}, X Théorème 3.1) pour voir que le lieu où $Y'\rightarrow X'$ est ramifié est de pure codimension $1$. Mais le lemme d'Abhyankar (\cite{SGA1}, X Lemme 3.6) montre que ce morphisme n'est pas non plus ramifié en codimension $1$. Il est donc étale, et comme $X'$ est strictement local, c'est un revêtement trivial de $X'$. 

 Comme $Y'\rightarrow Y$ est séparé (car affine), la section définie par le morphisme $Y\rightarrow X'$ ci-dessus est une immersion fermée. L'égalité des dimensions, et le fait que $Y$ et $Y'$ sont réduits (car normaux), impliquent que cette section identifie $Y$ à une composante irréductible de $Y'$. Donc le morphisme $Y\rightarrow X'$ construit au départ est en fait un isomorphisme au dessus de $X$. Le groupe $G=\Aut_X Y$ s'identifie canoniquement à $\Aut_X X'={\boldsymbol \mu_{\bold r}}$, et on conclut grâce au corollaire \ref{champracchampquot2} (ou plutôt, grâce à sa preuve, qui donne une version explicite de l'isomorphisme).

\end{proof}

\begin{lem}
\label{defc}
 Soit $Y\rightarrow X$ dans un objet de $\Rev^D (X)$, $Z\rightarrow X$ un objet galoisien de $\Rev^D (X)$, de groupe $G$, dominant $Y\rightarrow X$, $\bold r$ les indices de ramification de $Z\rightarrow X$, $H$ le groupe de Galois de $Z\rightarrow Y$. Le morphisme de champs $[Z| H]\rightarrow [Z| G]$ est étale, et composé avec l'isomorphisme canonique $[Z| G]\simeq \sqrt[\bold r]{\bold D/X}$ défini dans la proposition \ref{abh} il définit un objet de $\varinjlim_{\bold r}\Cat \Rev (\sqrt[\bold r]{\bold D/X})$, qui est, à isomorphisme près, indépendant du choix de $Z$.  
\end{lem}

\begin{defi}
  On notera $C(Y\rightarrow X)$ l'objet de $\varinjlim_{\bold r}\Cat \Rev (\sqrt[\bold r]{\bold D/X})$ défini dans le lemme \ref{defc}.
\end{defi}

Il est clair qu'on a en fait défini un foncteur 

$$C:\Rev^D (X) \rightarrow \varinjlim_{\bold r}\Cat \Rev (\sqrt[\bold r]{\bold D/X})$$

\subsection{Le foncteur $M$}
\label{sec:le-foncteur-m}

\begin{lem}
\label{norm}
Soit  $\bold r=(r_i)_{i\in I}$ une famille 
d'entiers non divisibles par la caractéristique $p$ de $k$
  et $T\rightarrow \sqrt[\bold r]{\bold D/X}$ un revêtement étale.
Soit $N(T)$ la fermeture intégrale de $X$ dans $R(T)/R(\sqrt[\bold r]{\bold D/X})=R(X)$.
\begin{enumerate}
\item[(i)] Il existe un unique morphisme de champs $T\rightarrow N(T)$ faisant commuter le diagramme

\xymatrix@R=2pt{
&&& T\ar@{.>}[r] \ar[dd]& N(T)\ar[dd]\\
&&& &\\
&&&  \sqrt[\bold r]{\bold D/X}\ar[r]&  X       }

et ce morphisme est surjectif.

\item[(ii)] Le morphisme canonique $N(T)\rightarrow X$ est un revêtement 
modérément ramifié de $X$ le long de $D$, et le foncteur obtenu 

\xymatrix@R=2pt{
&&& \Rev (\sqrt[\bold r]{\bold D/X})\ar[r]&\Rev^D (X)}

commute au produit fibré.

\item[(iii)] Si de plus  $T\rightarrow \sqrt[\bold r]{\bold D/X}$ est galoisien de groupe $G$, $N(T)\rightarrow X$ l'est également, de multi-indice de ramification $\bold r'$ divisant $\bold r$, et le morphisme 
$N(T)\rightarrow [N(T)|G]\simeq \sqrt[\bold r']{\bold D/X}$ défini dans le lemme \ref{abh} s'inscrit dans un diagramme cartésien :

\xymatrix@R=2pt{
&&& T\ar[r] \ar[dd]\ar@{}[rdd]|{\Box}& N(T)\ar[dd]\\
&&& &\\
&&&  \sqrt[\bold r]{\bold D/X}\ar[r]&  \sqrt[\bold r']{\bold D/X}    }
\end{enumerate}
\end{lem}

\begin{proof}

\begin{enumerate}
\item[\it{(i)}] Pour l'existence et l'unicité du morphisme : soit $T_0\rightarrow T$ un atlas étale, et  $s,b:T_1\rightrightarrows T_0$ le groupoïde correspondant. Il résulte de \cite{GM} Proposition 1.8.5 et \cite{SGA1} Exposé I, Corollaire 9.10, que $T_1$ et $T_0$ sont normaux. L'affirmation résulte alors de la propriété universelle de la fermeture intégrale (\cite{SGA2} 6.3.9).

	Pour montrer la surjectivité, on peut supposer $X=\spec R$, où $R$ est un anneau. On pose $R'=\frac{R[(T_i)_{i\in I}]}{(T_i^{r_i}-s_i)_{i\in I}}$, où les $(T_i)_{i\in I}$ sont des indéterminées, et $X'=\spec R'$ (quitte à renommer $I$, on peut supposer qu'aucun des $s_i$ n'est inversible). D'après le corollaire \ref{champracchampquot2} il existe un isomorphisme naturel de champs sur $X$ : $\sqrt[\bold r]{\bold D/X}     \simeq [X'|\boldsymbol{\mu}_{\bold r} ]$. 
Posons $T'=X'\times_{[X'|\boldsymbol{\mu}_{\bold r} ]}T$, c'est un atlas étale de $T$. De plus $T'\rightarrow X$ est fini (car $T'\rightarrow T$ et $T\rightarrow X$ le sont) et $N(T)\rightarrow X$ est séparé (car affine), et donc $T'\rightarrow N(T)$ est fini, et en particulier entier. Comme ce morphisme est de plus dominant (car $T'\rightarrow T$ est surjectif, et $T\rightarrow N(T)$ est birationnel), il est surjectif, d'après le théorème de Cohen-Seidenberg.

\item[\it{(ii)}]
$N(T)\rightarrow X$ est modérément ramifié le long de $D$ : il faut vérifier les cinq conditions de la Définition 2.2.2 de \cite{GM}.

\begin{enumerate}
\item[1)] $N(T)\rightarrow X$ est fini : ça résulte de la finitude de la fermeture intégrale d'un anneau noethérien normal dans une extension séparable finie (\cite{BourbAlgComV}, V, Proposition 18, Corollaire 1).
\item[2)] $N(T)\rightarrow X$ est étale au dessus de $U=X-D$ : posons $V=T\times_{\sqrt[\bold r]{\bold D/X}}U$, c'est un revêtement étale de $U$, et comme $U$ est normal, $V$ s'identifie d'après \cite{EGA4.4}, Corollaire 18.10.12.à la fermeture intégrale de $U$ dans $R(V)=R(T)$, qui n'est autre que $N(T)\times_X U$. 

\item[3)] Toute composante irréductible de $N(T)$ domine $X$ : 
les composantes irréductibles de $N(T)$ sont les mêmes que celle de son ouvert dense $V$, or $V\rightarrow U$ étant étale, toute composante irréductible de $V$ domine $U$.
\item[4)] $N(T)$ est normal : clair d'après \cite{SGA2} Proposition 6.3.7, ou simplement d'après la propriété universelle définissant $N(T)$. 
\item[5)] Pour tout point générique $x$ de $D$, $N(T)$ est modérément ramifié au dessus de $\mathcal O_{X,x}$ : c'est également clair d'après $\it{(i)}$, qui permet d'appliquer \cite{GM} Lemma 2.2.5, vu qu'un revêtement de Kummer est modérément ramifié (\cite{GM} Example 2.2.4). 
	On en tire également la propriété sur les indices de ramification du $\it{(iii)}$.
\end{enumerate}

	Le foncteur obtenu commute au produit fibré : donnés $T\rightarrow \sqrt[\bold r]{\bold D/X}$,  $T'\rightarrow \sqrt[\bold r]{\bold D/X}$ deux revêtements étales, on a un morphisme canonique $N(T\times_{\sqrt[\bold r]{\bold D/X}}T')\rightarrow N(T)\times_X N(T')$ (produit fibré schématique), comme ce morphisme est entier et birationnel, et $N(T\times_{\sqrt[\bold r]{\bold D/X}}T')$ est normal, il identifie $N(T\times_{\sqrt[\bold r]{\bold D/X}}T')$ à la normalisation de $N(T)\times_X N(T')$, qui par définition est le produit fibré de $N(T)$ et $N(T')$ sur $X$ dans $\Rev_D(X)$.

\item[\it{(iii)}]

	La première assertion résulte de la seconde affirmation de $\it{(ii)}$.
Pour la seconde assertion, il suffit de remarquer que le morphisme naturel
$T\rightarrow \sqrt[\bold r]{\bold D/X}\times_{\sqrt[\bold r']{\bold D/X}}N(T)$ est un morphisme birationnel de revêtements étales de $\sqrt[\bold r]{\bold D/X}$, c'est donc un isomorphisme d'après \cite{EGA4.4}, Corollaire 18.10.12.

\end{enumerate}
\end{proof}

Le lemme \ref{norm} permet de poser :
\begin{defi}
  On définit un foncteur $M: \varinjlim_{\bold r}\Cat \Rev (\sqrt[\bold r]{\bold D/X})\rightarrow \Rev^D (X)$ sur les objets en posant, pour une famille $\bold r=(r_i)_{i\in I}$ d'entiers non divisibles par la caractéristique $p$ de $k$ : $M(T\rightarrow\sqrt[\bold r]{\bold D/X})=N(T)\rightarrow X$. 
\end{defi}

\subsection{Conclusion}
\label{sec:conclusion}

\begin{proof}[Preuve de la proposition \ref{grfundmodgrfundchamp}]
  
Il suffit de montrer que le foncteur $C$ est une équivalence. 

Le fait qu'un revêtement modéré est par définition normal permet de définir une transformation naturelle $1\rightarrow MC$, et la propriété universelle de la normalisation montre que c'est un isomorphisme.

Le lemme \ref{norm} {\it(ii)} permet à son tour de définir une transformation naturelle $ (T\rightarrow \sqrt[\bold r]{\bold D/X})\rightarrow CM(T\rightarrow \sqrt[\bold r]{\bold D/X})$ dans $\varinjlim_{\bold r}\Cat \Rev (\sqrt[\bold r]{\bold D/X})$, et le point {\it(iii)} montre que c'est un isomorphisme pour $T\rightarrow \sqrt[\bold r]{\bold D/X}$ galoisien, et on en déduit que c'est vrai pour tout objet. Donc $C$ et $M$ sont des équivalences de catégories réciproques.

\end{proof}

\section{Faisceaux localement constants et fibrés finis sur un champ de Deligne-Mumford}
\label{sec:fibres-finis}

Dans cette partie, on fixe à nouveau un champ de Deligne-Mumford localement noethérien $X$.

\subsection{Topologies}
\label{sec:topologies}

La définition suivante est une adaptation\footnote{je remercie l'auteur pour avoir fourni le fichier source de son texte, me permettant ainsi de reproduire les diagrammes} de \cite{Zoon}, Lemme 1.2.

\begin{defi}
  On définit le \emph{site étale} $X_{et}$ (resp. le site étale fini $X_{etf}$) 
du champ $X$ comme le site dont la catégorie sous-jacente a pour objets les morphismes représentables étales $f:T\rightarrow X$ d'un champ de Deligne-Mumford vers $X$, a pour flèches les classes d'isomorphismes de couples $(\phi, \alpha)$

$$ \shorthandoff{;:!?}
  \xymatrix{ 
    T \ar[rr]^\phi \ar[rd]_f|*{}="A"
    && T' \ar@{=>}"A"_-\alpha \ar[ld]^{f'} \\ & X
    }$$
où $\phi$ est un morphisme représentable, $\alpha$ un $2$-isomorphisme,
et dont les recouvrements sont les familles épimorphiques (resp. les familles épimorphiques $(T_i\rightarrow T)_{i\in I}$ constituée de morphismes \emph{finis}). 

\end{defi}

On parlera aussi de topologie étale pour $X_{et}$, et de topologie étale finie globale pour $X_{etf}$.

\begin{rem}[\cite{Zoon} Lemme 1.2]
\label{remtop}
  On obtient une topologie équivalente à $X_{et}$ si dans la définition des objets, on impose à $T$ d'être un schéma.
\end{rem}

\subsection{Systèmes locaux ensemblistes et groupe fondamental}
\label{sec:syst-loca-ensembl}

Zoonekynd (\cite{Zoon}) a remarqué que l'on pouvait interpréter la définition 
\ref{defgrfundchamp} comme un cas particulier du groupe fondamental d'un topos donnée par Leroy (\cite{Leroy}). 

\subsubsection{Systèmes locaux ensemblistes}
\label{sec:syst-loca-ensembl-1}

\begin{defi}[\cite{Leroy},\cite{Zoon}]
 Donné un topos $\mathcal T$, on définit la sous-catégorie $\LC(\mathcal T)$
 (resp. $\LCF(\mathcal T)$) comme celle des objets localement constants
 (resp. localement constants finis) et $\SLC(\mathcal T)$
 (resp. $\SLCF(\mathcal T)$) comme celle des unions disjointes d'objets de 
$\LC(\mathcal T)$ (resp. de $\LCF(\mathcal T)$).
  \end{defi}

$\SLC(\mathcal T)$ (resp. $\SLCF(\mathcal T)$) est un topos galoisien
(resp. un topos galoisien fini).
Si on suppose $\mathcal T$ connexe, et que l'on fixe un point géométrique 
$x$ de $\mathcal T$, on associe canoniquement à cette catégorie un progroupe strict $\pi_1(\mathcal T,x)$ (resp. un groupe profini $\hat{\pi_1}(\mathcal T,x)$), vérifiant $\SLC(\mathcal T)\simeq \pi_1(\mathcal T,x)-\Ens$ 
(resp. $\SLCF(\mathcal T)\simeq \hat{\pi_1}(\mathcal T,x)-\Ens$), les ensembles étant munis d'action continues des pro-groupes considérés. De plus $(\LCF(\mathcal T),x^*)$ est une catégorie galoisienne dont le groupe fondamental s'identifie à $\hat{\pi_1}(\mathcal T,x)-\Ens$.

\subsubsection{Systèmes locaux ensemblistes et revêtements}
\label{sec:syst-loca-ensembl-2}

On dispose d'un foncteur naturel $\Cat\Rev X\rightarrow \LCF(\widetilde{X_{et}})$ donné sur les objets par $(Y\rightarrow X)\rightarrow \Hom_X(\cdot, Y)$, où 
$\Hom_X(\cdot, Y)$ est donné comme foncteur sur les objets par $(T\rightarrow X)\rightarrow \Hom_T(T,T\times_Y X)$. Pour voir que cela définit bien un préfaisceau d'ensembles, on peut, en utilisant la remarque \ref{remtop}, supposer que $T$ est un schéma, mais ça résulte alors du fait que $Y\rightarrow X$ est représentable, donc que  $T\times_Y X$ est également un schéma. Le fait que  $\Hom_X(\cdot, Y)$ est effectivement un faisceau sur $X_{et}$ découle immédiatement du fait que c'est vrai lorsque $X$ est un schéma (\cite{SGA4}, VII.2, \cite{VistDesc}).

On dispose également d'un foncteur naturel dans la direction opposée $\LCF(\widetilde{X_{et}})\rightarrow \Cat\Rev X$. En effet, soit $p:X_0\rightarrow X$ un atlas étale, $X_1=X_0\times_X X_0$, et $s,b:X_1\rightrightarrows X_0$ le groupoïde correspondant. Soit $\bold E\in \obj\LCF(\widetilde{X_{et}})$, on peut l'interpréter (\cite{VistDesc}, Example 4.11) comme un faisceau localement constant fini équivariant sur $X_0$, c'est à dire un couple $(\bold F
, \psi)$, où $\bold F \in \obj\LCF(\widetilde{{X_0}_{et}})$, et $\phi : s^*\bold F\rightarrow b^*\bold F$ est un isomorphisme vérifiant la condition de cocycle habituelle. L'équivalence de catégories usuelle $\LCF(\widetilde{{X_i}_{et}})\simeq \Rev X_i$ pour $i\in \{0,1\}$, permet d'interpréter cette donnée comme un revêtement étale $Y_0\rightarrow X_0$, muni d'un isomorphisme $\phi:{X_1}_{\searrow s} \times_{X_0}Y_0 \simeq {X_1}_{\searrow b} \times_{X_0}Y_0$. En posant $Y_1= {X_1}_{\searrow s} \times_{X_0}Y_0$, on obtient un nouveau groupoïde $(pr_2, pr_2\circ\phi:Y_1\rightrightarrows Y_0)$ et l'on obtient ainsi un champ $Y=[Y_1\rightrightarrows Y_0]$ et même en fait un objet de $\Cat\Rev X$.

\begin{thm}[\cite{Zoon}, Théorème 3.1]
\label{revsysloc}
  Les foncteurs ci-dessus définissent des équivalence de catégories réciproques
$\Cat\Rev X\simeq \LCF(\widetilde{X_{et}})$.
\end{thm}

\begin{cor} 
Si $X$ est un champ de Deligne-Mumford connexe, et $x:\spec \Omega \rightarrow X$ un point géométrique, on a un isomorphisme naturel
  $$\hat{\pi_1}(\widetilde{X_{et}},x)\simeq \pi_1(X,x)$$
\end{cor}

\begin{rem}
 $ {\pi_1}(\widetilde{X_{et}},x)$ porte le nom de groupe fondamental élargi de $X$ (voir \cite{SGA3.2} X 7.6 pour le cas d'un schéma). 
\end{rem}

\subsubsection{Interprétation à l'aide de la topologie étale finie globale}
\label{sec:interpr-laide-de}
\begin{prop}
\label{systloc}
 Soit $X$ champ de Deligne-Mumford connexe. Le foncteur
 $\SLC(\widetilde{X_{etf}})\rightarrow\SLC(\widetilde{X_{et}})$ induit un
 isomorphisme de $\SLC(\widetilde{X_{etf}})$
 sur $\SLCF(\widetilde{X_{et}})$. En particulier, si $x$ est un point géométrique, on a un isomorphisme naturel 
$\hat{\pi_1}(\widetilde{X_{et}},x)\simeq \pi_1(\widetilde{X_{etf}},x)$.
\end{prop}

\begin{proof}
	Le morphisme de sites $f:X_{etf}\rightarrow X_{et}$ induit un morphisme de topos $(f^*,f_*):\widetilde{X_{et}}\rightarrow\widetilde{X_{etf}}$ dont l'adjoint à gauche $f^*$ induit un foncteur fidèlement plein $\SLC(\widetilde{X_{etf}})\rightarrow\SLC(\widetilde{X_{et}})$, et dont on va montrer que l'image essentielle est $\SLCF(\widetilde{X_{et}})$. Via l'équivalence $\SLC(\mathcal T)\simeq {\pi_1}(\mathcal T,x)-\Ens$, ce foncteur s'interprète comme la restriction le long du morphisme ${\pi_1}(\widetilde{X_{et}},x)\rightarrow \pi_1(\widetilde{X_{etf}},x)$.
	
Soit d'abord $\bold E\in \obj\LC(\widetilde{X_{etf}})$. Il existe une famille couvrante $(T_i\rightarrow X)_{i\in I}$ constituée de morphismes représentables étales finis telle que pour tout $i\in I$, $\bold E_{|T_i}$ soit constant.
L'image de  $T_i\rightarrow X$ est ouverte car le morphisme est étale, et fermé car il est fini, c'est donc $\emptyset$ ou $X$. On peut donc se ramener à une famille couvrante à un élément $T\rightarrow X$, revêtement qu'on peut supposer de plus galoisien. Soit $G$ son groupe de Galois. Alors $\bold E$ correspond (via l'équivalence $ \pi_1(\widetilde{X_{etf}},x)-\Ens\simeq \SLC(\widetilde{X_{etf}})$) à un $G$-ensemble $E$, qui se décompose en $E=\coprod_{j\in J} E_j$, ses orbites sous $G$. Chaque ensemble $E_j$ est fini car $G$ l'est, donc correspond (via l'équivalence $\pi_1(\widetilde{X_{et}},x)-\Ens\simeq \SLC(\widetilde{X_{et}})$) à un $\bold E_j\in \obj\LCF(\widetilde{X_{et}})$. Donc $\bold E= \coprod_{j\in J}\bold E_j$ s'envoie sur un objet de $\SLCF(\widetilde{X_{et}})$.

Réciproquement si $\bold E\in \obj\LCF(\widetilde{X_{et}})$, alors le théorème \ref{revsysloc} montre que $\bold E$ définit un revêtement $Y\rightarrow X$, dont une clôture galoisienne trivialise $\bold E$, donc $\bold E$ vient bien d'un objet de $\LC(\widetilde{X_{etf}})$.

\end{proof}

\begin{cor}
\label{intetf} 
Si $X$ est un champ de Deligne-Mumford connexe, et $x:\spec \Omega \rightarrow X$ un point géométrique, on a un isomorphisme naturel
  $$\pi_1(\widetilde{X_{etf}},x)\simeq \pi_1(X,x)$$
\end{cor}

\subsection{La catégorie tannakienne des systèmes locaux de $k$-vectoriels }
\label{sec:la-categ-tann}

\begin{defi}[\cite{Saav}, chapitre VI, 1.1.2]
  Donné un topos $\mathcal T$ connexe localement connexe, et un corps $k$, on définit la catégorie $\LC(\mathcal T,k)$ des systèmes locaux de $k$-vectoriels de rang fini.
 \end{defi}

Si on choisit de plus un point géométrique $x$,
c'est une catégorie tannakienne. Son groupe de Tannaka est alors l'enveloppe $k$-algébrique du progroupe strict $\pi_1(\mathcal T,x)$, au sens suivant.

\begin{prop}[\cite{Saav}, chapitre VI, 1.1.2.1]
\label{grtann}
  Si $\pi_1(\mathcal T,x)=(G_i)_{i\in I}$ et $H_i$ est l'enveloppe $k$-algébrique de $G_i$, alors le groupe de Tannaka de  $(\LC(\mathcal T,k),x^*)$ est canoniquement isomorphe à $\varprojlim_{i \in I}H_i$.
\end{prop}

On déduit du corollaire \ref{intetf} et de la proposition \ref{grtann} :

\begin{cor}
\label{tanprof}
Soit $X$ un champ de Deligne-Mumford connexe, et $x$ un
point géométrique. Le groupe de Tannaka de $(\LC(\widetilde{X_{etf}},k),x^*)$ 
est canoniquement isomorphe au groupe fondamental profini $\pi_1(X,x)$.
\end{cor}

\begin{rem}
  \begin{enumerate}
  \item Il vaudrait mieux ici parler du $k$-groupe proconstant associé à $\pi_1(X,x)$.
\item Le groupe de Tannaka de $(\LC(\widetilde{X_{et}},k),x^*)$ est, d'après ce qui précède, isomorphe l'enveloppe $k$-algébrique du groupe fondamental élargi de $X$.
  \end{enumerate}
   \end{rem}

 \begin{lem}
    
\label{lemutil} Soit $X$ champ de Deligne-Mumford connexe.
Si $\bold V\in \obj\LC(\widetilde{X_{etf}},k)$, alors il existe un revêtement $Y\rightarrow X$ de $X$ trivialisant $\bold V$. 
  \end{lem}

  \begin{proof}
 C'est immédiat à partir du corollaire \ref{tanprof}.
  \end{proof}

 \subsection{Foncteur à la Riemann-Hilbert}
\label{sec:foncteur-la-riemann}

\subsubsection{Définition}
\label{sec:definition-1}

Donné un champ de Deligne-Mumford $X$, on peut définir la catégorie $\Vect X$
des fibrés vectoriels sur $X$ comme la catégorie $[X,\Vect]$ des morphismes de
champs de $X$ vers le champ $\Vect$ des fibrés vectoriels, parfois appelés
représentations du champ $X$, c'est le point de vue que l'on a adopté jusqu'à présent.

La théorie de la descente des fibrés vectoriels (et plus
généralement des faisceaux quasi-cohérents \cite{SGA1}) fournit un point de vue
alternatif, en effet les faisceaux de $\mathcal O_X$-modules $\mathcal F$ sur
$X_{et}$, tels qu'il existe un atlas étale $X'\rightarrow X$, tel que $\mathcal
F_{|X'}$ est libre, forment une catégorie équivalente (\cite{LMB}, chapitre 13). On utilisera librement cette équivalence par la suite. 
La définition suivante est inspirée de \cite{Saav} VI 1.2.4.

\begin{defi}
  Soit $X$ un champ de Deligne-Mumford localement noethérien sur un corps $k$.
On définit le foncteur à la Riemann-Hilbert 

$$\RH:\LC(\widetilde{X_{etf}},k)\rightarrow   \Vect X$$

comme le foncteur composé du foncteur canonique $\LC(\widetilde{X_{etf}},k)\rightarrow\LC(\widetilde{X_{et}},k)$ et du foncteur $\LC(\widetilde{X_{et}},k)\rightarrow  \Vect X$ donné sur les objets par $\bold V\rightarrow \mathcal O_X\otimes_k \bold V$.

\end{defi}

\subsubsection{Propriétés du foncteur $\RH$}
\label{sec:propr-du-fonct}

\begin{prop}
\label{proprRH}  
Soit $X$ un champ de Deligne-Mumford localement noethérien sur un corps $k$.

  \begin{enumerate}
  \item Le foncteur $\RH$ est fidèle.
\item Si $X$ est de plus complet, réduit, et $k$ est algébriquement clos, il est fidèlement plein.
  \end{enumerate}
\end{prop}

\begin{proof}
Pour $\bold V\in \obj  \LC(\widetilde{X_{etf}},k)$, on note $\phi_{X,\bold V}: \bold V\rightarrow \mathcal O_X\otimes_k \bold V$ le morphisme de faisceaux sur $X_{etf}$ défini à partir du morphisme canonique $\bold k \rightarrow \mathcal O_X$.
Quitte à remplacer $\bold V$ par $\Homb(\bold V,\bold W)$, il suffit de voir :

\begin{enumerate}
\item $H^0(X,\phi_{X,\bold V})$ est injectif.
\item Si $X$ est complet et réduit sur $k$ algébriquement clos, $H^0(X,\phi_{X,\bold V})$ est bijectif.
\end{enumerate}

Le premier point est évident car $\bold k \rightarrow \mathcal O_X$ est injectif, et les foncteurs $\cdot\otimes_k \bold V$ et
  $H^0(X,\cdot)$ sont exacts à gauche.

Pour le second point, $X$ étant localement noethérien (ce qui assure que les composantes connexes sont ouvertes), 
on peut supposer $X$ connexe. Le lemme \ref{lemutil} donne l'existence d'un revêtement $\pi :Y\rightarrow X$ trivialisant $\bold V$. On peut supposer $\pi$ galoisien de groupe $G$. 

On a alors un diagramme commutatif :

\xymatrix@R=12pt{
0\ar[r]&H^0(X,\bold V)\ar[r]^{\pi^{-1}}\ar[dd]^{H^0(X,\phi_{X,\bold V})}& H^0(Y,\pi^{-1}\bold V)\ar[dd]^{H^0(Y,\phi_{Y,\pi^{-1}\bold V})}\ar@<1ex>[r]^{pr_1^{-1}}\ar@<-1ex>[r]_{pr_2^{-1}}& H^0(Y\times_X Y,p^{-1}\bold V)\ar[dd]^{H^0(Y\times_X Y,\phi_{Y,p^{-1}\bold V})}\\
&&&\\
0\ar[r]&H^0(X,\mathcal O_X\otimes_k \bold V)\ar[r]^{\pi^*}& H^0(Y,\mathcal O_Y\otimes_k  \pi^{-1}\bold V)\ar@<1ex>[r]^{pr_1^*}\ar@<-1ex>[r]_{pr_2^*}&H^0(Y\times_X Y,\mathcal O_{Y\times_X Y}\otimes_k  p^{-1}\bold V)
} 

où $p$ désigne le morphisme canonique $p:Y\times_X Y \rightarrow X$.

$Y$ étant encore propre (car fini sur $X$) et réduit (\cite{SGA1} I Proposition 9.2) on s'est donc ramené au cas où $\bold V$ est trivial. 
On peut à nouveau supposer $X$ connexe, et donc $\bold V=s_X^{-1}V$, où $s_X: X\rightarrow \spec k$ est le morphisme structurel, et $V$ un $k$-vectoriel de rang fini. On est alors immédiatement ramené à $V=k$, et il s'agit de voir que le morphisme naturel $k\rightarrow H^0(X,\mathcal O_X)$ est un isomorphisme, mais ça résulte du fait que $X$ est propre, réduit, connexe, et $k$ algébriquement clos.   

\end{proof}

\subsection{Fibrés finis}
\label{sec:fibres-finis-1}

\begin{defi}
\label{schtord}
  On appellera \emph{schéma tordu} un champ de Deligne-Mumford $X$ admettant pour espace des modules un schéma $M$, tel qu'il existe un ouvert dense 
$U$ de $M$, tel que $X\rightarrow M$ soit un isomorphisme en restriction à $U$. 
\end{defi}

On adapte les définitions de \cite{NoriRFG}, \cite{NoriFGS} au cas d'un schéma
tordu $X$ modéré (au sens de \cite{AV}, définition 2.3.2) réduit sur un corps $k$, dont l'espace des modules $M$ est propre et connexe sur $k$.

\begin{defi}[\cite{NoriRFG}, \cite{NoriFGS}]
Un faisceau localement libre $\mathcal E$ sur $X$ est dit \emph{fini}
s'il existe deux polynômes distincts $P,Q$ à coefficients entiers positifs
tels que $P(\mathcal E)\simeq Q(\mathcal E)$. 
\end{defi}

Pour identifier l'image essentielle du foncteur $RH$, on va suivre la
stratégie de Nori, qui consiste à plonger la catégorie des fibrés finis dans
la catégorie abélienne des fibrés semi-stables sur $X$.

\begin{defi}
	\label{orbicourbe}
  Une \emph{orbicourbe} dans $X$ est un morphisme birationnel
  sur son image $\sqrt[\bold r]{\bold D/C}\rightarrow X$, où $C$ est une
  courbe projective, connexe, et lisse sur $k$, $\bold D=(D_i)_{i\in I}$ un
  ensemble de diviseurs de Cartier effectifs réduits sur $C$.
\end{defi}

\begin{defi}[\cite{NoriRFG}, \cite{NoriFGS}]
\label{defsemistab}
Un faisceau localement libre $\mathcal E$ sur $X$ est dit \emph{semi-stable}
s'il est semi-stable de degré $0$ en restriction à toute orbicourbe dans $X$. 
On notera $\SS_0 X$ la sous-catégorie pleine de $\Vect X$ des faisceaux localement libres semi-stables sur $X$. 
\end{defi}

\begin{prop}
\label{ssab}
 La catégorie $\SS_0 X$ est une catégorie abélienne.
\end{prop}

\begin{proof}
La preuve est identique à celle de \cite{NoriRFG}, Lemma 3.6, (b) : étant
donné un morphisme $f:\mathcal E\rightarrow \mathcal E'$ dans $\SS_0 X$, le
point clé est de voir que $\ker f$ et $\coker f$ sont localement
libres. Il est aisé de voir qu'ils sont sans torsion \footnote{les arguments généraux de \cite{Simpson}, \S 3 s'appliquent ici, à l'aide de \cite{Borne}, \S 5 pour les adapter au cas des orbicourbes ; comme c'est par ailleurs bien connu dans le cadre -équivalent- des fibrés paraboliques sur les courbes (voir \cite{Seshadri}), nous ne rentrons pas dans les détails }, et donc localement libres si $X$ est une orbicourbe. Dans le cas général, $X$ étant réduit cela revient à voir que la fonction qui à un point géométrique $x:\spec \overline k\rightarrow X$ associe le rang de $x^*f :x^* \mathcal E\rightarrow x^*\mathcal E'$ est constant sur $X$. Or, le cas particulier envisagé ci-dessus montre que cette fonction est constante sur toute orbicourbe dans $X$. On peut donc conclure à l'aide du lemme suivant :

\begin{lem}
  La relation d'équivalence sur les points $x:\spec \overline
  k\rightarrow X$ engendrée par $x\sim x'$ s'il existe une orbicourbe dans
  $X$ dont l'image contient $x$ et $x'$ admet une unique orbite.
\end{lem}

\begin{proof}
  Dans le cas où $X$ est un schéma, on se ramène au cas où $X$ est projectif
  sur $k$ grâce au lemme de Chow (\cite{SGA2}, 5.6), où c'est un fait classique.

Dans le cas général, on note $U$ comme dans la définition \ref{schtord} un ouvert de l'espace de modules $M$ de $X$ tel que 
la flèche $X\rightarrow M$ de $X$ vers son espace de
modules $M$ soit un isomorphisme en restriction à $U$. Soit $y$, $y'$ les images respectives de $x$, $x'$ dans $M$, on
peut supposer que $y\in U$. D'après le cas particulier ci-dessus, il existe
une courbe $C$ dans $M$ (au sens de \cite{NoriRFG}, ou de la définition \ref{orbicourbe}) contenant $y$ et $y'$.
Le champ de Deligne-Mumford $C\times_M X$ admet $C$ pour espace des modules, et il en est de même de  $(C\times_M X)_{red}$ (\cite{AV} Lemma 2.3.3 ou \cite{AOV} Corollary 3.3). D'après \cite{Cadman}, Theorem 4.1, on a un isomorphisme $(C\times_M X)_{red}\simeq \sqrt[\bold r]{\bold D/C}$ sur $X$ pour un choix convenable d'une famille $\bold D$ de diviseurs effectifs et d'une famille d'entiers naturels $\bold r$. Comme on dispose d'un morphisme birationnel surjectif $ \sqrt[\bold r]{\bold D_{red}/C}\rightarrow \sqrt[\bold r]{\bold D/C}$, comme de plus $(C\times_M X)_{red}$ est birationnel sur son image dans $X$ et que celle-ci contient $x$ et $x'$, on a terminé.

\end{proof}
 
 \end{proof}

 \begin{prop}
   Tout fibré fini sur $X$ est semi-stable. 
 \end{prop}

 \begin{proof}
   Comme la restriction d'un fibré fini l'est encore, il suffit de le vérifier
   sur les orbicourbes. Mais on peut alors adapter la preuve de \cite{NoriRFG}
 au cas des orbicourbes : voir \cite{Borne}, Proposition 6. 
 \end{proof}

\begin{defi}[\cite{NoriRFG}, \cite{NoriFGS}]
Un faisceau localement libre $\mathcal E$ sur $X$ est dit
\emph{essentiellement fini} si c'est un quotient de deux sous-fibrés semi-stables d'un fibré fini. On notera $\EF X$ la sous-catégorie pleine de $\SS_0 X$ des
faisceaux localement libres essentiellement finis sur $X$. 
\end{defi}

\begin{thm}
\label{eftan}
Soit $X$ un schéma tordu modéré et réduit sur un corps $k$, dont l'espace des modules $M$ est propre et connexe sur $k$, et $x\in X(k)$ un point rationnel. 
La paire $(\EF X, x^*)$ est une catégorie tannakienne.
\end{thm}

\begin{proof}
  Compte tenu de la proposition \ref{ssab}, la preuve est la même que celle
  donnée dans \cite{NoriRFG}, \S3.
\end{proof}

\begin{cor}
\label{coreftan}
 Si on suppose, en plus des hypothèses du théorème \ref{eftan}, que $k$ est 
algébriquement clos de caractéristique $0$, alors tout fibré essentiellement
fini est fini, et le foncteur $\RH$ induit une équivalence de catégories
tensorielles entre  $\LC(\widetilde{X_{etf}},k)$ et $\F X$. En particulier $(\F X,x^*)$ est
une catégorie tannakienne dont le groupe est canoniquement isomorphe à $\pi_1(X,x)$.
\end{cor}

\begin{proof}
	Soit $\bold V\in \obj \LC(\widetilde{X_{etf}},k)$. Le fait que $\RH(\bold V)$ soit fini résulte du lemme \ref{lemutil} : si $\pi : Y\rightarrow X$ est un revêtement galoisien de groupe $G$ trivialisant $\bold V$, il existe une représentation $V$ de $G$ sur le corps $k$ telle que $\pi^{-1}\bold V=V_Y$. On suit alors l'argument de \cite{NoriRFG} Proposition 3.8 : cette représentation se plonge dans un $k[G]$-module libre et est donc essentiellement finie, comme la caractéristique de $k$ est $0$, elle en en fait finie. Or le morphisme naturel $\RH(\bold V)\rightarrow \pi_*^G(\mathcal O_Y\otimes_k V)$ est un isomorphisme, et donc   $\RH(\bold V)$ est lui-même fini. 

La proposition \ref{proprRH} montre que le foncteur $RH$ est fidèlement plein.

Soit à présent $\mathcal E$ un fibré essentiellement fini. Soit $<\mathcal E>$ la sous-catégorie tannakienne engendrée, et $G$ son groupe de Tannaka (i.e. le
groupe d'holonomie de $\mathcal E$). Comme $\mathcal E$ est essentiellement
fini, $G$ est un schéma en groupe fini sur $k$ (\cite{NoriRFG} Theorem 1.2), comme ce corps est de caractéristique
$0$, $G$ est réduit d'après un théorème de Cartier (\cite{Water}, Chapter 11),
donc étale, et $k$ étant de plus algébriquement clos, $G$ est donc constant. 

Le foncteur tensoriel $G\Rep\rightarrow \Vect X$ correspond d'après
\cite{NoriRFG} Proposition 2.9 \footnote{qui, du fait de sa fonctorialité, vaut aussi pour les champs de Deligne-Mumford, voir  sur le sujet \cite{Lurie}} à un $G$-revêtement $\pi : Y\rightarrow X$, et $\mathcal E$
s'identifie via l'équivalence  $<\mathcal E>\simeq G\Rep$ à une
représentation $V$ de $G$. Si $\bold V$ est le système local correspondant, on a vu ci-dessus que  $\RH(\bold V)$ est isomorphe à $\pi_*^G(\mathcal O_Y\otimes_k V)$, lui-même isomorphe à $\mathcal E$. D'où les deux premières assertions.

  La dernière résulte alors du corollaire \ref{tanprof}.
\end{proof}

\section{Théorème de Weil-Nori}
\label{sec:theoreme-de-weil}

\subsection{Fibrés paraboliques modérés}
\label{sec:fibr-parab-moderes} 

Soit $X$ un schéma localement noethérien sur un corps $k$, $\bold D$ une famille de diviseurs irréductibles à croisements normaux simples sur $X$. 

\subsubsection{Fibrés paraboliques finis}
\label{sec:fibr-parab-finis}

\begin{defi}

  \begin{enumerate}
  \item On définit la catégorie $\Par(X,\bold D)$ des fibrés paraboliques modérés sur $(X,\bold D)$ par :
$$\Par(X,\bold D)= \varinjlim_{\bold r} \Par_{\frac{1}{\bold r}}(X,\bold D)$$
où les multi-indices varient parmi les familles $\bold r=(r_i)_{i\in I}$ d'entiers non divisibles par la caractéristique $p$ de $k$.
\item $\Par(X,\bold D)$ est munie d'un produit tensoriel vérifiant, pour $\mathcal E_\cdot, \mathcal E'_\cdot \in \obj \Par_{\frac{1}{\bold r}}(X,\bold D)$, la formule de convolution suivante :

$$(\mathcal E_\cdot\otimes\mathcal E'_\cdot)_{\bold m}=\int^{\bold l\in \frac{1}{\bold r}\mathbb Z} \mathcal E_{\bold l}\otimes \mathcal E'_{\bold m-\bold l}$$

où $\int$ désigne la cofin (coend), voir \S \ref{sec:oper-elem-sur}.

\item  Un fibré parabolique modéré $\mathcal E_\cdot$ 
sur $(X,\bold D)$ est dit \emph{fini} s'il existe deux polynômes distincts $P,Q$ à coefficients entiers positifs tels que $P(\mathcal E_\cdot)\simeq Q(\mathcal E_\cdot)$. On notera $\FPar(X,\bold D)$ la catégorie des fibrés paraboliques modérés sur $(X,\bold D)$.

\end{enumerate}
\end{defi}

\begin{rem}
\label{comp}
  Ces notions sont compatibles, via l'équivalence $\Vect(\sqrt[\bold r]{\bold D/X})\simeq \Par_{\frac{1}{\bold r}}(X,\bold D)$ du théorème \ref{fibparfibchamp}, avec les notions champêtres du \S \ref{sec:fibres-finis-1}, voir \cite{Borne}.
\end{rem}

\subsubsection{Fibrés paraboliques essentiellement finis}
\label{sec:fibr-parab-essent-1}

\begin{defi}

  \begin{enumerate}
  \item

Un fibré parabolique modéré $\mathcal E_\cdot$ sur $(X,\bold D)$ à poids
multiples de ${\frac{1}{\bold r}}$ est dit
semi-stable si le faisceau localement libre sur $\sqrt[\bold r]{\bold D/X}$
associé par la correspondance du théorème \ref{fibparfibchamp} est semi-stable
au sens de la définition \ref{defsemistab}. On notera $\SS_0\Par(X,\bold D)$ la sous-catégorie pleine de $\Par(X,\bold D)$ dont les objets sont semi-stables. 

\item Un fibré parabolique modéré semi-stable $\mathcal E_\cdot$ est dit 
\emph{essentiellement fini} si c'est un quotient de deux sous-fibrés paraboliques modérés semi-stables d'un fibré parabolique modéré fini. On notera $\EFPar(X,\bold D)$ la sous-catégorie pleine de $\SS_0\Par(X,\bold D)$ dont les objets sont essentiellement finis.
 
  \end{enumerate}
\end{defi}

\begin{rem}
  \begin{enumerate}
  \item La définition de semi-stabilité est indépendante du choix de $\bold
    r$.
\item Il serait intéressant de donner une définition de la semi-stabilité 
ne faisant intervenir que la topologie de Zariski.
  \end{enumerate}
\end{rem}

\subsection{Lien avec le groupe fondamental}
\label{sec:lien-avec-le}

\subsubsection{Énoncé}
\label{sec:enonce}

\begin{thm}
\label{thmfinal}
  Soit $X$ un schéma propre, normal, connexe sur un corps $k$, $\bold D$ une famille de diviseurs irréductibles à croisements normaux simples sur $X$, $D=\cup_{i\in I}D_i$, $x\in X(k)$ un point rationnel, $x\notin D$.

  \begin{enumerate}
  \item[(i)] La paire $(\EFPar(X,\bold D),x^*)$ est une catégorie tannakienne.
  \item[(ii)] Si $k$ est algébriquement clos de caractéristique $0$, tout fibré parabolique modéré essentiellement fini est fini, et le groupe de Tannaka de 
$(\FPar(X,\bold D),x^*)$ est canoniquement isomorphe au groupe fondamental $\pi_1(X-D,x)$.
 \end{enumerate}
 
\end{thm}

\begin{proof}
  On commence par remarquer que $\sqrt[\bold r]{\bold D/X}$ est normal (\cite{GM}, Proposition 1.8.5). 
 
  \begin{enumerate}
  \item[{\it(i)}] Ceci résulte alors des théorèmes \ref{fibparfibchamp} et \ref{eftan}.
  \item[{\it(ii)}] La première assertion découle du théorème \ref{fibparfibchamp} et du corollaire \ref{coreftan}. 
Pour la seconde, notons $\pi$ ce schéma en groupe. Alors $\pi\simeq
\varprojlim_{\bold r}\pi_{\bold r}$, où $\pi_{\bold r}$ est le groupe de
Tannaka de la catégorie $(\FPar_{\frac{1}{\bold r}}(X,\bold D),x^*)$, avec des
notations évidentes. D'après la proposition \ref{grfundmodgrfundchamp}, il
suffit de voir qu'on a des isomorphismes naturels $\pi_{\bold
  r}\simeq\pi_1(\sqrt[\bold r]{\bold D/X},x)$, compatibles avec les systèmes
projectifs (vu qu'on est en caractéristique zéro, on a un isomorphisme naturel $\pi_1(X-D,x)\simeq
\pi_1^D(X,x)$ donné, au niveau des revêtements, par le foncteur de normalisation). On conclut donc en appliquant à nouveau le théorème
\ref{fibparfibchamp} et le corollaire \ref{coreftan}.
  \end{enumerate}

\end{proof}

\subsubsection{Schéma en groupe fondamental modéré}
\label{sec:schema-en-groupe}

On est naturellement conduit à poser :

\begin{defi}
  Avec les notations du théorème \ref{thmfinal}, on appellera schéma en groupe fondamental modéré de $(X,D)$ le groupe fondamental $\pi^D(X,x)$ de la catégorie tannakienne $(\EFPar(X,\bold D),x^*)$.
\end{defi}

\begin{rem}
  \begin{enumerate}
  \item Ce schéma en groupe $\pi^D(X,x)$ est une limite inverse de schémas en groupes finis, se spécialise sur le schéma en groupe fondamental de Nori (\cite{NoriRFG}) lorsque $D=\emptyset$, et sur le groupe fondamental modéré de Grothendieck-Murre (\cite{GM}) lorsque $k$ est algébriquement clos de caractéristique $0$, d'où son nom.
  \item Lorsque $k$ est quelconque, les arguments de corollaire \ref{coreftan} montrent qu'on a un morphisme $\pi^D(X,x)\rightarrow \pi_1^D(X,x)$ qui est un épimorphisme lorsque $k$ est algébriquement clos.
\item Toutefois, il conviendrait de préciser la nature des ``torseurs modérément ramifiés'' que $\pi^D(X,x)$ classifie.
  \end{enumerate}
\end{rem}

\section{Application au calcul de fibrés paraboliques finis de groupe d'holonomie résoluble}
\label{calholres}
\subsection{Introduction et notations}
\label{intrnotcal}

On reprend les hypothèses de la partie \ref{sec:theoreme-de-weil}: $X$ est un schéma propre, normal, connexe sur un corps $k$, qu'on suppose de plus algébriquement clos de caractéristique $0$, $\bold D$ une famille de diviseurs irréductibles à croisements normaux simples sur $X$, $D=\cup_{i\in I}D_i$, $x\in X(k)$ un point rationnel, $x\notin D$. Le but de cette partie est d'utiliser le théorème \ref{thmfinal} pour construire explicitement certains objets de $\FPar(X,\bold D)$. 

Puisqu'on est en fait intéressé par le groupe fondamental modéré de $X-D$, on évite bien sûr de construire un tel fibré parabolique fini à partir d'un revêtement de $Y\rightarrow X$ modérément ramifié le long de $D$ le trivialisant. L'idée de la méthode présentée est de n'utiliser que des sous-revêtements d'un tel $Y\rightarrow X$, et repose essentiellement sur la proposition \ref{imfibrfin}. Cette observation est inspirée par une transcription directe de la méthode des petits groupes de Wigner et Mackey de la théorie des représentations à la théorie des revêtements, rendue possible grâce au théorème \ref{thmfinal}. 

Par la suite, on note $\mathcal X=\sqrt[\bold r]{\bold D/X}$ le champ des racines.

\subsection{Compléments sur les fibrés finis}

On commence par quelques remarques générales concernant les fibrés finis. Comme la structure parabolique n'entre pas vraiment en jeu, ce qui va être dit est aussi valable dans la situation classique où $\mathcal X$ est un schéma, propre, réduit et connexe sur un corps $k$ algébriquement clos de caractéristique $0$. Dans cette situation, il suffit de remplacer l'utilisation du théorème \ref{thmfinal} par le théorème de Nori originel \cite{NoriRFG}.

\subsubsection{Image directe d'un fibré fini}

La base de la méthode pour construire des fibrés finis est la remarque élémentaire suivante :
\begin{prop}
	\label{imfibrfin}
	Soit $p:\mathcal Y\rightarrow \mathcal X$ un revêtement étale. Si $\mathcal F\in \obj F\mathcal Y$, alors $p_*\mathcal F \in \obj F\mathcal X$.
\end{prop}

\begin{proof}
	On peut choisir $\mathcal Y$ connexe, et aussi un point $y\in \mathcal Y(k)$ au dessus de $x$.

	Soit $G$ un groupe fini, $B_kG=[\spec k| G]$ le champ classifiant, $\pi_1^{et}(\mathcal X,x)\rightarrow G$ un morphisme, $\mathcal Z\rightarrow \mathcal X$ le revêtement galoisien associé, correspondant aussi à un morphisme $m: \mathcal X\rightarrow B_kG$. Alors le foncteur $\RH$ induit une équivalence entre la catégorie $G\Rep$ des représentations de $G$ et la catégorie $F_{\mathcal Z}\mathcal X$ des fibrés finis sur $\mathcal X$ trivialisés par $\mathcal Z\rightarrow \mathcal X$. Si $V$ est une représentation de $G$, et $\bold V$ est le système local sur $\mathcal X_{et}$ associé, alors cette correspondance associe à $V$ le fibré $\mathcal O_{\mathcal X}\otimes_k\bold V$, qui est canoniquement isomorphe à $m^* V$.
	
On choisit à présent un morphisme $\pi_1^{et}(\mathcal Y,y)\rightarrow A$, tel que le revêtement galoisien associé $\mathcal Z\rightarrow \mathcal Y$ trivialise $\mathcal F$, et tel que le revêtement composé $\mathcal Z\rightarrow \mathcal X$ soit galoisien, de groupe $G$. D'après ce qui précède, il existe une représentation $W$ du groupe $A$ telle que $\mathcal F\simeq \mathcal O_{\mathcal Y}\otimes_k\bold W$. La proposition résulte alors du 

\begin{lem}
	$p_* (\mathcal O_{\mathcal Y}\otimes_k\bold W)\simeq  \mathcal O_{\mathcal X}\otimes_k \bold V $ où $V=\Ind_A^G W$.
\end{lem}

\begin{proof}
	Il s'agit d'une formule de changement de base dans le diagramme cartésien :
		\begin{displaymath}
			\xymatrix{
			\mathcal Y \ar[r]\ar[d] & B_k A \ar[d] \\
			\mathcal X\ar[r] & B_k G}
                \end{displaymath}

\end{proof}

\end{proof}

Dans la pratique, il est utile de savoir calculer le produit tensoriel de deux fibrés finis obtenus par la méthode de la proposition \ref{imfibrfin}. Pour cela, il suffit d'adapter à ce contexte les formules classiques, dues à Mackey, donnant le produit tensoriel de deux représentations induites comme somme directe de représentations induites (voir par exemple \cite{CurtisReiner} \S 44). On obtient ainsi :

\begin{lem}
	\label{prodtensimfibrfin}
	Soit $\mathcal Z \rightarrow \mathcal X$ un revêtement galoisien connexe, $G$ le groupe de Galois, $H_1$, $H_2$ deux sous-groupes. Pour $i\in \{ 1,2\}$, on note $p_i:\mathcal Y_i\rightarrow \mathcal X$ le revêtement intermédiaire correspondant au sous-groupe $H_i$, et $\mathcal F_i\in \obj F_{\mathcal Z}\mathcal Y_i$. De plus pour $g\in G$, on note $p_g : \mathcal Y_g\rightarrow \mathcal X$ le revêtement correspondant au sous-groupe $H_1\cap g H_2 g^{-1}$ et $q_{g,i}:\mathcal Y_g\rightarrow \mathcal Y_i$ le morphisme naturel. Alors
	\begin{displaymath}
		{p_1}_*\mathcal F_1 \otimes_{\mathcal O_{\mathcal X}} {p_2}_*\mathcal F_2 \simeq \oplus_{g\in H_1\backslash G/H_2} {p_g}_*(q_{g,1}^* \mathcal F_1 \otimes_{\mathcal O_{\mathcal X_g}} q_{g,2}^* \mathcal F_2)\end{displaymath}
	 
\end{lem}

\begin{proof}
	On a un isomorphisme naturel de $B_k G$-groupoïdes $B_kH_1\times_{B_k G} B_k H_2 \simeq \coprod_{H_1\backslash G/H_2} B_k(H_1\cap g H_2 g^{-1})$. En le tirant par le morphisme $\mathcal X \rightarrow B_k G$ définissant $\mathcal Z \rightarrow \mathcal X$, on obtient un $\mathcal X$-isomorphisme $\mathcal Y_1\times_{\mathcal X} \mathcal Y_2\simeq  \coprod_{H_1\backslash G/H_2} \mathcal Y_g$. De plus si $p: \mathcal Y_1\times_{\mathcal X} \mathcal Y_2 \rightarrow \mathcal X$ est le morphisme canonique, la formule de changement de base donne 
${p_1}_*\mathcal F_1 \otimes {p_2}_*\mathcal F_2\simeq p_*(\pr_1^*\mathcal F_1\otimes \pr_2^*\mathcal F_2)$, d'où la formule annoncée.  
\end{proof}

\subsubsection{La méthode des petits groupes de Wigner et Mackey}

Comme application de la partie précédente, on décrit les fibrés finis associés à une extension triviale d'un revêtement galoisien par un groupe abélien.

On se donne donc un revêtement galoisien connexe $\mathcal Z\rightarrow \mathcal X$ de groupe $G=A\ltimes H$, où $H$ est quelconque, et $A$ est abélien, d'exposant $n$ premier à l'ordre de $H$. On note $\mathcal Y\rightarrow \mathcal X$ le revêtement intermédiaire correspondant à $A$, et on fixe un point $z\in \mathcal Z(k)$ au dessus de $x$, d'image $y$ dans $\mathcal Y(k)$. D'après la dualité de Tannaka, on a une équivalence naturelle $\F_\mathcal Z \mathcal Y \simeq A\Rep$. En particulier le groupe $\Pic_\mathcal Z \mathcal Y$ des classes d'isomorphisme de fibrés inversibles sur $\mathcal Y$ trivialisés par $\mathcal Z$ est canoniquement isomorphe au groupe $\hat A$ des caractères de $A$, et détermine complètement $\F_\mathcal Z \mathcal Y$.

Le but est de décrire complètement la catégorie tannakienne $\F_\mathcal Z \mathcal X$ en fonction de $\F_\mathcal Y \mathcal X'$ (pour les extensions $\mathcal X'$ intermédiaires entre $\mathcal Y$ et $\mathcal X$) et de $\Pic_\mathcal Z \mathcal Y$. On commence par décrire la structure additive.

On remarque qu'on a une action naturelle de $H$ sur $\hat A\simeq \Pic_\mathcal Z \mathcal Y$. Alors l'inclusion $\Pic_\mathcal Z \mathcal Y \subset H^1_{et}(\mathcal Y, \boldsymbol{\mu}_n)$ est $H$-équivariante. Soit $\mathcal L$ un fibré inversible sur $\mathcal Y$ trivialisé par $\mathcal Z$ . On note $H_{\mathcal L}$ le stabilisateur de sa classe dans $\Pic_\mathcal Z \mathcal Y$, et $\pi_{\mathcal L} : \mathcal Y\rightarrow \mathcal Y/ H_{\mathcal L}$ le morphisme quotient. Comme celui-ci est étale, on dispose de la suite spectrale de Hochschild-Serre, qui s'écrit ici : $$H^p(H_{\mathcal L},H^q(\mathcal Y,\boldsymbol{\mu}_n))\Longrightarrow H^{p+q}(\mathcal Y/ H_{\mathcal L}, \boldsymbol{\mu}_n)$$ L'hypothèse que $n$  est premier à l'ordre de $H$ et la suite exacte des termes de bas degré associée à la suite spectrale montrent que $H^1_{et}(\mathcal Y/ H_{\mathcal L},\boldsymbol{\mu}_n)\simeq H^1_{et}(\mathcal Y, \boldsymbol{\mu}_n)^{H_{\mathcal L}}$, et donc il existe, à isomorphisme près, un unique fibré inversible de $n$-torsion $\tilde{\mathcal L}$ sur $\mathcal Y/ H_{\mathcal L}$ tel que $\mathcal L \simeq \pi_{\mathcal L}^*   \tilde{\mathcal L} $. On note de plus $p_{\mathcal L}: \mathcal Y/ H_{\mathcal L}\rightarrow \mathcal X$  le morphisme canonique.

\begin{prop}
	\begin{enumerate}
		\item  Soit $\mathcal L$ un fibré inversible sur $\mathcal Y$ trivialisé par $\mathcal Z$ et $\mathcal E \in \obj\F_\mathcal Y (\mathcal Y/H_{\mathcal L})$. Le fibré ${p_{\mathcal L}}_*(\tilde{\mathcal L}\otimes \mathcal E)$ sur $\mathcal X$ est fini.
		\item Lorsque $\mathcal L$ varie dans un système de représentants de $(\Pic_\mathcal Z \mathcal Y)/H$ et $\mathcal E$ varie dans une base de générateurs additifs de $\F_\mathcal Y (\mathcal Y/ H_{\mathcal L})$, ces faisceaux forment une base de générateurs additifs de $\F_\mathcal Z \mathcal X$. 

\end{enumerate}

	\end{prop}

	\begin{proof}
		\begin{enumerate}
			\item C'est une application directe de la proposition \ref{imfibrfin}.
			\item C'est, en fait, via la correspondance de Tannaka entre fibrés finis et représentations du groupe fondamental (corollaire \ref{coreftan}), un problème de théorie des groupes, pour lequel on renvoie à \cite{SerreRLGF}, 8.2, Proposition 25. 
		\end{enumerate}
	
	\end{proof}

	La structure tensorielle de $\F_\mathcal Z \mathcal X$ est alors complètement déterminée par le lemme \ref{prodtensimfibrfin}.

	\subsubsection{Fibrés finis de groupe d'holonomie résoluble}

En poursuivant la même idée, on voit que la méthode conduit au calcul des fibrés finis dont le groupe d'holonomie (i.e. le groupe de Tannaka de la catégorie tannakienne engendrée) est résoluble.

En effet, soit $p:\mathcal Y \rightarrow \mathcal X$ un revêtement connexe, étale, galoisien de groupe d'automorphismes $H$, et $y\in \mathcal Y(k)$ un point au dessus de $x$. On note $\PIC(\mathcal Y)[n]$ la sous-catégorie pleine de $F\mathcal Y$ dont les objets sont les fibrés inversibles de $n$-torsion, où $n\geq 1$ est un entier. 

Donné un groupe (abstrait, ou profini) $\pi$, on note $D_n(\pi)$ le noyau du morphisme de groupe $\pi\rightarrow \frac{\pi^{ab}}{n}$, c'est un sous-groupe caractéristique.

\begin{lem}
	\label{holres}
	Le groupe de Tannaka de la catégorie tannakienne engendrée par l'image du foncteur $p_*: \PIC(\mathcal Y)[n]\rightarrow F\mathcal X$ est canoniquement isomorphe à $\frac{\pi_1(\mathcal X,x)}{D_n(\pi_1(\mathcal Y,y))}$. 
\end{lem}

\begin{proof}
	Le morphisme canonique $\pi_1(\mathcal X,x)\rightarrow \frac{\pi_1(\mathcal X,x)}{D_n(\pi_1(\mathcal Y,y))}$ correspond à un nouveau revêtement étale, galoisien, connexe $\mathcal Y'\rightarrow \mathcal X$, muni d'un point géométrique $y'\in \mathcal Y'(k)$ au dessus de $x$, et dominant $\mathcal Y \rightarrow \mathcal X$. Par dualité de Tannaka, la catégorie $\Rep\frac{\pi_1(\mathcal X,x)}{D_n(\pi_1(\mathcal Y,y))}$ est canoniquement isomorphe à la catégorie $F_{\mathcal Y'}\mathcal X$ des fibrés finis sur $\mathcal X$ trivialisés par $\mathcal Y'$. Or, soit $\mathcal E$ un tel fibré, $p^*\mathcal E$ est un objet de $F_{\mathcal Y'}\mathcal Y$, dont le groupe de Tannaka est isomorphe à $\frac{\pi_1(\mathcal Y,y)^{ab}}{n}$, donc est abélien et de $n$-torsion. On peut donc écrire 
	$p^*\mathcal E\simeq \oplus_{i=1}^N \mathcal L_i$, où les $\mathcal L_i$ sont dans $\PIC(\mathcal Y)[n]$. Mais comme $k$ est supposé de caractéristique $0$, $\mathcal E\simeq p_*^Hp^*\mathcal E$ est un facteur direct de $p_*p^*\mathcal E\simeq \oplus_{i=1}^N p_*\mathcal L_i$.
\end{proof}

On garde les notations de la preuve, en particulier $\mathcal Y'$ est le plus grand revêtement abélien $n$-élémentaire de $\mathcal Y$. Soit de plus $A$ le dual de Cartier du groupe $\Pic^0(\mathcal Y)[n]$. La théorie de Kummer usuelle affirme que $\frac{\pi_1(\mathcal Y,y)^{ab}}{n}\simeq A$. L'avantage de la méthode utilisée ici, qui peut être vue comme une version relative de la théorie de Kummer, est qu'elle donne une interprétation tannakienne du groupe $G=\frac{\pi_1(\mathcal X,x)}{D_n(\pi_1(\mathcal Y,y))}$ : si l'on sait calculer la catégorie tannakienne engendrée par l'image du foncteur $p_*: \PIC(\mathcal Y)[n]\rightarrow F\mathcal X$, on sait déterminer $G$ comme extension de $H$ par $A$.

On peut en particulier itérer le procédé, en partant de $(\mathcal Y,y)=(\mathcal X,x)$, et en choisissant une suite d'entiers $n_1,\cdots,n_m$, ce qui conduit au calcul des quotients $\frac{\pi_1(\mathcal X,x)}{D_{n_m}\cdots D_{n_1}(\pi_1(\mathcal X,x))}$. La limite naturelle de la méthode est bien le plus grand quotient pro-résoluble $\pi_1^{res}(\mathcal X,x)$ de $\pi_1(\mathcal X,x)$, vu qu'on obtient ainsi un sous-ensemble cofinal de l'ensemble de ses quotients finis. 

	\subsection{Fibrés paraboliques finis de groupe d'holonomie résoluble}

	\subsubsection{Fibrés paraboliques finis obtenus comme image directe le long d'un morphisme modérément ramifié}

On conserve les notations de la partie \ref{intrnotcal}.

Soit $p:Y\rightarrow X$ dans $\obj \Rev^D (X)$, avec $Y$ connexe. On note $(E_j)_{j\in J}$ la famille des composantes irréductibles (munies de la structure réduite) de $p^{-1}(D)$. Soit $Z\rightarrow X$ une clôture galoisienne, $(r_i)_{i\in I}$ (respectivement $(s_j)_{j\in J}$) la famille des indices de ramification de la famille $(D_i)_{i\in I}$ (respectivement $(E_j)_{j\in J}$) dans $Z$.

\begin{lem}
	Le morphisme naturel $q: \sqrt[\bold s]{\bold E|Y} \rightarrow \sqrt[\bold r]{\bold D|X}$ est fini étale.
\end{lem}

\begin{proof}
	Ceci résulte du lemme d'Abhyankar. En effet si $G$ (respectivement $A$) est le groupe de Galois de $Z\rightarrow X$ (respectivement $Z\rightarrow Y$) le lemme \ref{abh} montre que $q$ s'identifie au morphisme entre champs quotients $[Z|A]\rightarrow [Z|G]$, qui est fini étale, car obtenu par changement de base à partir du morphisme des champs classifiants $B_kA\rightarrow B_kG$ par le morphisme $[Z|G]\rightarrow B_k G$ correspondant au $G$-torseur $Z\rightarrow [Z|G]$.
\end{proof}

On peut donc utiliser la proposition \ref{imfibrfin} pour construire des fibrés paraboliques finis. De plus, ces fibrés finis sont explicitement calculables\footnote{dans la mesure où l'on sait déterminer l'image directe d'un fibré vectoriel usuel par $p$, mais ce problème peut-être pris en charge par le théorème de Grothendieck-Riemann-Roch.} grâce à la proposition \ref{imdirfibpar}.

Plus précisément, si on veut calculer des fibrés paraboliques finis de groupe d'holonomie résoluble, on peut appliquer le lemme \ref{holres} aux champs des racines. Ainsi, il est naturel d'essayer de déterminer la $n$-torsion $\Pic(\sqrt[\bold r]{\bold D/X})[n]$ du groupe de Picard des champ des racines, pour $n\geq 1$ entier, ce qui est l'objet de la partie suivante.

\subsubsection{Fibrés inversibles de torsion sur les champs des racines}

Le foncteur de Picard des champs algébriques a été récemment étudié par S.Brochard (voir \cite{Brochard}). Il a en particulier montré qu'on pouvait en étudier la composante neutre comme dans le cas classique des schémas. On rappelle brièvement les définitions dont on aura besoin.

\begin{defi}
	\begin{enumerate}
		\item Deux faisceaux inversibles $\mathcal L$ et $\mathcal L'$ sur le champ $\mathcal X$ sont dits \emph{algébriquement équivalents} s'ils sont équivalents pour la relation d'équivalence engendrée par la relation : $\mathcal L \sim \mathcal L'$ s'il existe un $k$-schéma connexe de type fini $T$, des points géométriques $t,t':\spec \Omega \rightarrow T$, un faisceau inversible $\mathcal M$ sur $\mathcal X\times_k T$, et des isomorphismes $(\mathcal L\times_k T)_{|\mathcal X_t}\simeq \mathcal M_{|\mathcal X_t}$, $(\mathcal L'\times_k T)_{|\mathcal X_{t'}}\simeq \mathcal M_{|\mathcal X_{t'}}$.
		\item On note $\Pico \mathcal X$ le sous-groupe des éléments $[\mathcal L]$ de $\Pic \mathcal X$ tels que $\mathcal L$ est algébriquement équivalent à $\mathcal O_\mathcal X$.
		\item On appelle groupe de Néron-Severi le groupe $\NS(\mathcal X)= \Pic \mathcal X/\Pico \mathcal X$.
		\item Si $T(A)$ désigne la torsion du groupe abélien $A$, on note $$\Pict\mathcal X=\ker(\Pic\mathcal X\rightarrow \NS(\mathcal X)/T(\NS(\mathcal X))$$
	\end{enumerate}
\end{defi}

\begin{lem}
	\begin{enumerate}
		\item $\Pic X\cap \Pico \sqrt[\bold r]{\bold D|X}=\Pico X$
                \item $\Pic X\cap \Pict \sqrt[\bold r]{\bold D|X}=\Pict X$
\end{enumerate}
\end{lem}

\begin{proof}
	\begin{enumerate}
		\item On note comme d'habitude $\pi :  \sqrt[\bold r]{\bold D|X}\rightarrow X$ le morphisme vers l'espace des modules, et $\mathcal X=\sqrt[\bold r]{\bold D|X}$. 
			Soient $\mathcal L$, $\mathcal L'$ deux faisceaux inversibles sur $X$ tels que $\pi^* \mathcal L \sim \pi^*\mathcal L'$, $T$ un $k$-schéma connexe de type fini, $t,t':\spec \Omega \rightarrow T$ des points géométriques, $\mathcal M$ un faisceau inversible sur $\mathcal X\times_k T$, et des isomorphismes $(\pi^*\mathcal L\times_k T)_{|\mathcal X_t}\simeq \mathcal M_{|\mathcal X_t}$, $(\pi^*\mathcal L'\times_k T)_{|\mathcal X_{t'}}\simeq \mathcal M_{|\mathcal X_{t'}}$. Alors $(\pi\times_k T)_* \mathcal M$ est un faisceau inversible sur $X\times_k T$ : en effet il est localement libre comme image directe d'un faisceau localement libre par un morphisme fini et plat, et comme $\pi$ est génériquement un isomorphisme, son rang est $1$. En appliquant les formules de changement de base pour un morphisme affine ou le long d'un morphisme plat on obtient des isomorphismes naturels	$(\mathcal L\times_k T)_{|X_t}\simeq ((\pi\times_k T)_* \mathcal M)_{|X_t}$, $(\mathcal L'\times_k T)_{|X_{t'}}\simeq ((\pi\times_k T)_* \mathcal M)_{|X_{t'}}$, ce qui montre que $\mathcal L \sim \mathcal L'$.

		\item C'est une conséquence claire du premier point et des définitions. 
	\end{enumerate}
\end{proof}

On a donc des monomorphismes canoniques

\begin{displaymath}
	\frac{\Pico \sqrt[\bold r]{\bold D|X}}{\Pico X}\hookrightarrow 
	\frac{ \Pict \sqrt[\bold r]{\bold D|X}}{\Pict X}\hookrightarrow  
	\frac{\Pic\sqrt[\bold r]{\bold D|X}}{\Pic X}
\end{displaymath}

On rappelle (corollaire \ref{picchamprac}) qu'on a un isomorphisme canonique $\frac{\Pic (\sqrt[\bold r]{\bold D|X})}{\Pic X}\simeq \prod_{i\in I} \frac{\mathbb Z}{r_i}$.

On note 
\begin{displaymath}
	\left(\prod_{i\in I} \frac{\mathbb Z}{r_i}\right)^0\subset 
	\left(\prod_{i\in I} \frac{\mathbb Z}{r_i}\right)^{\tau}\subset 
	\prod_{i\in I} \frac{\mathbb Z}{r_i}
\end{displaymath}

les sous-groupes associés aux images des monomorphismes ci-dessus.

Comme $X$ est propre sur le corps algébriquement clos $k$, le groupe $\NS(X)$ est de type fini (\cite{SGA6}, XIII, Théorème 5.1), ce qui prouve qu'il en est de même pour $\NS(\sqrt[\bold r]{\bold D|X})$ et donne un sens à l'énoncé suivant.

\begin{prop}
	\label{torspicorb}
Soit $n$ premier à $\# T(\NS(X))$.
Il y a une suite exacte naturelle :
$$0\rightarrow \Pic(X)[n]\rightarrow \Pic(\sqrt[\bold r]{\bold D|X})[n]\rightarrow \left(\prod_{i\in I} \frac{\mathbb Z}{r_i}\right)^{\tau}[n]\rightarrow 0$$
\end{prop}

\begin{proof}

D'après les définitions, on a clairement $\Pict(X)[n]=\Pic(X)[n]$ et $\Pict(\sqrt[\bold r]{\bold D|X})[n]=\Pic(\sqrt[\bold r]{\bold D|X})[n]$. 

Pour conclure, on doit justifier que $\Ext^1_\mathbb Z(\mathbb Z/n,\Pict(X))=0$. Or (voir par exemple \cite{Weibel} 3.3.2) 
$\Ext^1_\mathbb Z(\mathbb Z/n,\Pict(X))\simeq \Pict(X)/n\Pict(X)$, et on peut conclure en appliquant le foncteur $\frac{\mathbb Z}{n}\otimes_\mathbb Z\cdot$ à la suite exacte :

\begin{displaymath}
0\rightarrow \Pico(X)\rightarrow	\Pict(X)\rightarrow T(\NS(X))\rightarrow 0
\end{displaymath}

En effet $\Pic^0(X)$ est un groupe divisible : comme $X/k$ est complet le foncteur de Picard $\SCPIC_{X/k}$ est représentable par un schéma (\cite{Murre}), comme $X$ est de plus normal $(\SCPIC^0_{X/k})_{red}$ est une variété abélienne sur $k$ (\cite{FGA}, 236, Corollaire 3.2), et on peut appliquer \cite{MumAV} II, \S 6, Application 2.

\end{proof}

\begin{rem}
	On a une interprétation assez directe en termes de revêtements : si $Y\rightarrow X$ est le revêtement étale galoisien de groupe abélien $n$-élémentaire maximal pour ces propriétés, et de même $Z\rightarrow X$ en remplaçant étale par modérément ramifié de multi-indice divisant $\bold r$, alors par dualité de Tannaka on voit que la suite exacte duale (pour la dualité de Cartier) de la suite exacte précédente est isomorphe à la suite exacte des groupes de Galois de la tour $Z\rightarrow Y\rightarrow X$.
\end{rem}

\subsubsection{Un exemple explicite}

L'exemple le plus simple possible de construction de fibré parabolique fini indécomposable de rang plus grand que $1$ est le suivant.

On considère le morphisme $p:Y=\mathbb P^1 \rightarrow X=\mathbb P^1$ donné par l'équation $y^2=\frac{x}{x-1}$, il est modérément ramifié le long du diviseur $(0)+(1)$, avec indices de ramification $2$. $p$ induit un morphisme fini étale $q:Y\rightarrow \sqrt[(2,2)]{(0,1)|X}$, et par changement de base on en déduit un morphisme fini étale $r: \sqrt[(3,3)]{(1,-1)|Y}\rightarrow \sqrt[(2,2,3)]{(0,1,\infty)|X}$.

A présent la proposition \ref{torspicorb} montre que 
$\Pico(\sqrt[(3,3)]{(1,-1)|Y})\simeq \left(\frac{\mathbb Z}{3}\times \frac{\mathbb Z}{3}\right)^{\tau}$, 
et on voit facilement que  
$$\left(\frac{\mathbb Z}{3}\times \frac{\mathbb Z}{3}\right)^0=
\left(\frac{\mathbb Z}{3}\times\frac{\mathbb Z}{3}\right)^{\tau}=
 \ker\left(\Sigma :\frac{\mathbb Z}{3}\times \frac{\mathbb Z}{3}\rightarrow \frac{\mathbb Q}{\mathbb Z}\right)$$
 où $\Sigma$ désigne la somme, donc $\Pic^0(\sqrt[(3,3)]{(1,-1)|Y})$ est cyclique d'ordre $3$. Pour $y\in \{1,-1\}$, soit $\mathcal N_y$ une racine cubique de $\mathcal O_Y((y))$ sur $\sqrt[(3,3)]{(1,-1)|Y}$, et soit $\mathcal L=\mathcal N_1^{\vee}\otimes \mathcal N_{-1}$. Alors le fibré $r_*(\mathcal L)$ est un fibré fini indécomposable de rang $2$ sur $\sqrt[(2,2,3)]{(0,1,\infty)|X}$.  

Le fibré parabolique correspondant a été considéré par L.Weng, voir \cite{LinWeng}, Appendix, \S 6, dans le langage du à Seshadri. Il nous semble que la description fournie ici en termes de champs des racines permet de préciser la description des drapeaux donnée par l'auteur, facilite le calcul de la structure tensorielle de la catégorie tannakienne engendrée (naturellement, dans ce cas précis, le groupe fondamental est isomorphe au groupe symétrique $\mathfrak S_3$), et enfin, donne une méthode de construction des fibrés paraboliques finis de groupe d'holonomie résoluble.

\subsubsection{Problème ouvert}

Dans \cite{BorneEmsalem} est donnée une preuve algébrique du théorème de structure du groupe fondamental pro-résoluble $\pi_1^{res}(X-D,x)$ pour $X$ une courbe projective et lisse sur un corps $k$ algébriquement clos de caractéristique $0$ , et $D$ un diviseur (non vide) sur celle-ci. Peut on utiliser le lemme \ref{holres} pour donner une preuve alternative ?

\appendix
\section{$2$-limite inductive filtrée de catégories}
\label{2limfil}
\subsection{$2$-limite}

On emprunte les définitions de \cite{hakim}. On note $\mathfrak{Cat}$ la $2$-catégorie dont les objets sont les catégories, les $1$-flèches les foncteurs, les $2$-flèches les transformations naturelles.

Soit $\mathfrak C$ une $2$-catégorie, $I$ une catégorie usuelle. 

Pour tout objet $\mathcal D$ de $\mathfrak C$, on note $c_{\mathcal D}$ 
le foncteur constant $I\rightarrow \mathfrak C$ envoyant tout objet de $I$ sur $\mathcal D$, et toute flèche de $I$ sur l'identité.

On note de plus $(I, \mathfrak C)$ la $2$-catégorie des pseudo-foncteurs de $I$ dans $\mathfrak C$.

Si $\mathcal F,\mathcal F' : I \rightarrow \mathfrak C$ sont deux pseudo-foncteurs, on note $(I,\mathfrak{C})(\mathcal F,\mathcal F')$ la catégorie des transformations naturelles de pseudo-foncteurs entre $\mathcal F$ et $\mathcal F'$.

\begin{defi}
	\label{deflimindcat}
Soit $\mathcal F : I \rightarrow \mathfrak C$ un pseudo-foncteur. On appelle \emph{$2$-limite inductive} de $\mathcal F$ le $2$-foncteur : 

\xymatrix@R=2pt{ 
&&&\mathfrak{C} \ar[r] & \mathfrak{Cat}\\
&&&\mathcal D \ar[r] & (I,\mathfrak{C})(\mathcal F,c_{\mathcal D})}

Si ce foncteur est $2$-représentable, on appelle aussi $2$-limite inductive et on note $\varinjlim_I \mathcal F(i)$ l'objet de $\mathfrak C$ le représentant.  
\end{defi}

Plus précisément, un représentant est un couple $(\mathcal C,\lambda)$, où $\mathcal C$ est un objet de $\mathfrak C$, et $\lambda : \mathcal F \rightarrow c_{\mathcal C}$ est une transformation naturelle entre pseudo-foncteurs, qui est $2$-universelle au sens suivant : pour tout objet $\mathcal D$ de $\mathfrak C$, et toute transformation naturelle $\mu : \mathcal F \rightarrow c_{\mathcal D}$, il existe un couple $(f,\theta)$ formé d'un $1$-morphisme $f:\mathcal C\rightarrow \mathcal D$ de $\mathfrak C$ et d'un $2$-isomorphisme $\theta$ de $(I,\mathfrak C)$ :

$$ \shorthandoff{;:!?}
  \xymatrix{ 
  \mathcal F \ar[rr]^\lambda \ar[rd]_\mu|*{}="A" && c_{\mathcal C} \ar@{=>}"A"_-\theta \ar[ld]^{c_f} \\ & c_{\mathcal D}
    }$$

qui est unique à $2$-isomorphisme unique près : si $(f',\theta')$ est un autre tel couple, il existe un unique $2$-isomorphisme $\rho$ de $\mathfrak C$ :

 \xymatrix{\relax
&&&&&     \mathcal C \UN[r]{f'}{f}{\rho} & \mathcal D}

tel que $\theta \circ (c_\rho\circ\lambda)=\theta'$.

\subsection{Cas des catégories}

On suppose désormais que la catégorie $I$ est filtrante et que $\mathfrak C=\mathfrak{Cat}$. Dans ce cas, on dispose d'une description naturelle et probablement bien connue de la $2$-limite d'un pseudo-foncteur $\mathcal F: I\rightarrow \mathfrak{Cat}$. Pour $f:i\rightarrow j$, on note $f_*:\mathcal F(i)\rightarrow \mathcal F(j)$, plutôt que $\mathcal C(f)$. 
Soit $\mathcal C$ la catégorie 

\begin{enumerate}

\item dont les objets sont les couples $(i,C)$, où $i$ est un objet de $I$, et $C$ est un objet de $\mathcal F(i)$,

\item dont les morphismes $(i,C)\rightarrow (j,D)$ sont les classes d'équivalence\footnote{on ignore délibérément les problèmes d'univers en omettant de vérifier qu'il s'agit d'un ensemble, pour être  correct, il faut probablement supposer la catégorie $I$ petite} de triplets $(f,g,\alpha)$ :

\xymatrix@R=2pt{ 
&&& i \ar[dr]^f &&&\\
&&&    &   k,  &    f_*C \ar[r]^{\alpha} & g_* D\\
&&& j \ar[ur]_g &&&
}

pour la relation

\xymatrix@R=2pt{ 
i \ar[dr]^f &&&&& i \ar[dr]^{f'} &&&\\
&   k,  &    f_*C \ar[r]^{\alpha} & g_* D&\sim &&   k',  &    f'_*C \ar[r]^{\alpha'} & g'_* D\\
j \ar[ur]_g &&&&& j \ar[ur]_{g'} &&&
}

s'il existe 

\xymatrix@R=2pt{ 
&&& k \ar[dr]^h &\\
&&&    &   l \\
&&& k' \ar[ur]_{h'} &
}

tel que $h\circ f=h'\circ f'$, $h\circ g=h'\circ g'$ et $h_*\alpha=h'_*\alpha'$.
\end{enumerate}
	
\begin{prop}
	La catégorie $\mathcal C$, munie de la transformation naturelle canonique $\mathcal F \rightarrow c_{\mathcal C}$, est une $2$-limite pour $\mathcal F$.
\end{prop}

\begin{proof}
	C'est une vérification longue mais directe de la définition \ref{deflimindcat} dans ce cas précis.

\end{proof}

\bibliography{SNC3}

\providecommand{\bibhead}[1]{}
\expandafter\ifx\csname pdfbookmark\endcsname\relax%
  \providecommand{\tocrefpdfbookmark}{}
\else\providecommand{\tocrefpdfbookmark}{%
   \hypertarget{tocreferences}{}%
   \pdfbookmark[1]{References}{tocreferences}}%
\fi

\tocrefpdfbookmark
\begin{thebibliography}{10}

\bibitem{SGA3.2}\bibhead{SGA3.2}
{\em Sch\'emas en groupes. {II}: {G}roupes de type multiplicatif, et structure
  des sch\'emas en groupes g\'en\'eraux}.
\newblock S\'eminaire de G\'eom\'etrie Alg\'ebrique du Bois Marie 1962/64 (SGA
  3). Dirig\'e par M. Demazure et A. Grothendieck. Lecture Notes in
  Mathematics, Vol. 152. Springer-Verlag, Berlin, 1962/1964.

\bibitem{SGA1}\bibhead{SGA1}
{\em Rev\^etements \'etales et groupe fondamental}.
\newblock Springer-Verlag, Berlin, 1971.
\newblock S\'eminaire de G\'eom\'etrie Alg\'ebrique du Bois Marie 1960--1961
  (SGA 1), Dirig\'e par Alexandre Grothendieck. Augment\'e de deux expos\'es de
  M. Raynaud, Lecture Notes in Mathematics, Vol. 224.
\newblock [\epfmt{arxiv}{math.AG/0206203}].

\bibitem{SGA6}\bibhead{SGA6}
{\em Th\'eorie des intersections et th\'eor\`eme de {R}iemann-{R}och}.
\newblock Springer-Verlag, Berlin, 1971.
\newblock S\'eminaire de G\'eom\'etrie Alg\'ebrique du Bois-Marie 1966--1967
  (SGA 6), Dirig\'e par P. Berthelot, A. Grothendieck et L. Illusie. Avec la
  collaboration de D. Ferrand, J. P. Jouanolou, O. Jussila, S. Kleiman, M.
  Raynaud et J. P. Serre, Lecture Notes in Mathematics, Vol. 225.

\bibitem{SGA4}\bibhead{SGA4}
{\em Th\'eorie des topos et cohomologie \'etale des sch\'emas. {T}ome 2}.
\newblock Springer-Verlag, Berlin, 1972.
\newblock S\'eminaire de G\'eom\'etrie Alg\'ebrique du Bois-Marie 1963--1964
  (SGA 4), Dirig\'e par M. Artin, A. Grothendieck et J. L. Verdier. Avec la
  collaboration de N. Bourbaki, P. Deligne et B. Saint-Donat, Lecture Notes in
  Mathematics, Vol. 270.

\bibitem{AGV}\bibhead{AGV}
{\sc Dan Abramovich, Tom Graber, and Angelo Vistoli}: {Gromov--Witten theory of
  Deligne--Mumford stacks}.
\newblock [\epfmt{arxiv}{math.AG/0603151}].

\bibitem{AOV}\bibhead{AOV}
{\sc Dan Abramovich, Martin Olsson, and Angelo Vistoli}: {Tame stacks in
  positive characteristic}.
\newblock [\epfmt{arxiv}{math.AG/0703310}].

\bibitem{AV}\bibhead{AV}
{\sc Alessandro Arsie and Angelo Vistoli}: Stacks of cyclic covers of
  projective spaces.
\newblock {\em Compos. Math.}, 140(3):647--666, 2004.

\bibitem{BBN}\bibhead{BBN}
{\sc Vikraman Balaji, Indranil Biswas, and Donihakkalu~S. Nagaraj}: Principal
  bundles over projective manifolds with parabolic structure over a divisor.
\newblock {\em Tohoku Math. J. (2)}, 53(3):337--367, 2001.

\bibitem{Biswasorb}\bibhead{Biswasorb}
{\sc Indranil Biswas}: Parabolic bundles as orbifold bundles.
\newblock {\em Duke Math. J.}, 88(2):305--325, 1997.

\bibitem{Borne}\bibhead{Borne}
{\sc Niels Borne}: Fibrés paraboliques et champ des racines.
\newblock {\em Int. Math. Res. Not.}, pp. article ID rnm049, 38, 2007.
\newblock [\epfmt{arxiv}{math.AG/0604458}].

\bibitem{BorneEmsalem}\bibhead{BorneEmsalem}
{\sc Niels Borne and Michel Emsalem}: {Note sur la détermination algébrique du
  groupe fondamental pro-résoluble d'une courbe affine}.
\newblock [\epfmt{arxiv}{0708.0093}].

\bibitem{BourbAlgComV}\bibhead{BourbAlgComV}
{\sc Nicolas Bourbaki}: {\em \'{E}l\'ements de math\'ematique. {F}asc. {XXX}.
  {A}lg\`ebre commutative. {C}hapitre 5: {E}ntiers. {C}hapitre 6:
  {V}aluations}.
\newblock Actualit\'es Scientifiques et Industrielles, No. 1308. Hermann,
  Paris, 1964.

\bibitem{Brochard}\bibhead{Brochard}
{\sc Sylvain Brochard}: Foncteur de picard d'un champ algébrique.
\newblock {\em Prépublication}, 2007.

\bibitem{Cadman}\bibhead{Cadman}
{\sc Charles Cadman}: Using stacks to impose tangency conditions on curves.
\newblock {\em Amer. J. Math.}, 129(2):405--427, 2007.
\newblock [\epfmt{arxiv}{math.AG/0312349}].

\bibitem{CurtisReiner}\bibhead{CurtisReiner}
{\sc Charles~W. Curtis and Irving Reiner}: {\em Representation theory of finite
  groups and associative algebras}.
\newblock Pure and Applied Mathematics, Vol. XI. Interscience Publishers, a
  division of John Wiley \& Sons, New York-London, 1962.

\bibitem{SGA2}\bibhead{SGA2}
{\sc Alexander Grothendieck}: \'{E}l\'ements de g\'eom\'etrie alg\'ebrique.
  {II}. \'{E}tude globale \'el\'ementaire de quelques classes de morphismes.
\newblock {\em Inst. Hautes \'Etudes Sci. Publ. Math.}, (8):222, 1961.

\bibitem{EGA3.1}\bibhead{EGA3.1}
{\sc Alexander Grothendieck}: \'{E}l\'ements de g\'eom\'etrie alg\'ebrique.
  {III}. \'{E}tude cohomologique des faisceaux coh\'erents. {I}.
\newblock {\em Inst. Hautes \'Etudes Sci. Publ. Math.}, (11):167, 1961.

\bibitem{FGA}\bibhead{FGA}
{\sc Alexander Grothendieck}: {\em Fondements de la g\'eom\'etrie alg\'ebrique.
  [{E}xtraits du {S}\'eminaire {B}ourbaki, 1957--1962.]}.
\newblock Secr\'etariat math\'ematique, Paris, 1962.

\bibitem{EGA4.1}\bibhead{EGA4.1}
{\sc Alexander Grothendieck}: \'{E}l\'ements de g\'eom\'etrie alg\'ebrique.
  {IV}. \'{E}tude locale des sch\'emas et des morphismes de sch\'emas. {I}.
\newblock {\em Inst. Hautes \'Etudes Sci. Publ. Math.}, (20):259, 1964.

\bibitem{EGA4.4}\bibhead{EGA4.4}
{\sc Alexander Grothendieck}: \'{E}l\'ements de g\'eom\'etrie alg\'ebrique.
  {IV}. \'{E}tude locale des sch\'emas et des morphismes de sch\'emas {IV}.
\newblock {\em Inst. Hautes \'Etudes Sci. Publ. Math.}, (32):361, 1967.

\bibitem{GM}\bibhead{GM}
{\sc Alexander Grothendieck and Jacob~P. Murre}: {\em The tame fundamental
  group of a formal neighbourhood of a divisor with normal crossings on a
  scheme}.
\newblock Springer-Verlag, Berlin, 1971.
\newblock Lecture Notes in Mathematics, Vol. 208.

\bibitem{hakim}\bibhead{hakim}
{\sc Monique Hakim}: {\em Topos annel\'es et sch\'emas relatifs}.
\newblock Springer-Verlag, Berlin, 1972.
\newblock Ergebnisse der Mathematik und ihrer Grenzgebiete, Band 64.

\bibitem{IyerSimpson}\bibhead{IyerSimpson}
{\sc Jaya N.~N. Iyer and Carlos~T. Simpson}: A relation between the parabolic
  {C}hern characters of the de {R}ham bundles.
\newblock {\em Math. Ann.}, 338(2):347--383, 2007.

\bibitem{LMB}\bibhead{LMB}
{\sc G{\'e}rard Laumon and Laurent Moret-Bailly}: {\em Champs alg\'ebriques},
  volume~39 of {\em Ergebnisse der Mathematik und ihrer Grenzgebiete. 3. Folge.
  A Series of Modern Surveys in Mathematics [Results in Mathematics and Related
  Areas. 3rd Series. A Series of Modern Surveys in Mathematics]}.
\newblock Springer-Verlag, Berlin, 2000.

\bibitem{Leroy}\bibhead{Leroy}
{\sc Olivier Leroy}: {\em Groupo\"\i de fondamental et th\'eor\`eme de van
  {K}ampen en th\'eorie des topos}, volume~17 of {\em Cahiers Math\'ematiques
  Montpellier [Montpellier Mathematical Reports]}.
\newblock Universit\'e des Sciences et Techniques du Languedoc U.E.R. de
  Math\'ematiques, Montpellier, 1979.

\bibitem{Lurie}\bibhead{Lurie}
{\sc Jacob Lurie}: {Tannaka Duality for Geometric Stacks}.
\newblock [\epfmt{arxiv}{math.AG/0412266}].

\bibitem{MacLane}\bibhead{MacLane}
{\sc Saunders MacLane}: {\em Categories for the working mathematician}.
\newblock Springer-Verlag, New York, 1971.
\newblock Graduate Texts in Mathematics, Vol. 5.

\bibitem{MY}\bibhead{MY}
{\sc Masaki Maruyama and Kôji. Yokogawa}: Moduli of parabolic stable sheaves.
\newblock {\em Math. Ann.}, 293(1):77--99, 1992.

\bibitem{MumAV}\bibhead{MumAV}
{\sc David Mumford}: {\em Abelian varieties}.
\newblock Tata Institute of Fundamental Research Studies in Mathematics, No. 5.
  Published for the Tata Institute of Fundamental Research, Bombay, 1970.

\bibitem{Murre}\bibhead{Murre}
{\sc Jacob~P. Murre}: On contravariant functors from the category of
  pre-schemes over a field into the category of abelian groups (with an
  application to the {P}icard functor).
\newblock {\em Inst. Hautes \'Etudes Sci. Publ. Math.}, (23):5--43, 1964.

\bibitem{Noohi}\bibhead{Noohi}
{\sc Behrang Noohi}: Fundamental groups of algebraic stacks.
\newblock {\em J. Inst. Math. Jussieu}, 3(1):69--103, 2004.
\newblock [\epfmt{arxiv}{math.AG/0201021}].

\bibitem{NoriRFG}\bibhead{NoriRFG}
{\sc Madhav~V. Nori}: On the representations of the fundamental group.
\newblock {\em Compositio Math.}, 33(1):29--41, 1976.

\bibitem{NoriFGS}\bibhead{NoriFGS}
{\sc Madhav~V. Nori}: The fundamental group-scheme.
\newblock {\em Proc. Indian Acad. Sci. Math. Sci.}, 91(2):73--122, 1982.

\bibitem{Saav}\bibhead{Saav}
{\sc Neantro Saavedra~Rivano}: {\em Cat\'egories {T}annakiennes}.
\newblock Springer-Verlag, Berlin, 1972.
\newblock Lecture Notes in Mathematics, Vol. 265.

\bibitem{SerreRLGF}\bibhead{SerreRLGF}
{\sc Jean-Pierre Serre}: {\em Repr\'esentations lin\'eaires des groupes finis}.
\newblock Hermann, Paris, revised edition, 1978.

\bibitem{Seshadri}\bibhead{Seshadri}
{\sc Conjeevaram~S. Seshadri}: {\em Fibr\'es vectoriels sur les courbes
  alg\'ebriques}, volume~96 of {\em Ast\'erisque}.
\newblock Soci\'et\'e Math\'ematique de France, Paris, 1982.
\newblock Notes written by J.-M. Drezet from a course at the \'Ecole Normale
  Sup\'erieure, June 1980.

\bibitem{Simpson}\bibhead{Simpson}
{\sc Carlos~T. Simpson}: Moduli of representations of the fundamental group of
  a smooth projective variety. {I}.
\newblock {\em Inst. Hautes \'Etudes Sci. Publ. Math.}, (79):47--129, 1994.

\bibitem{VistDesc}\bibhead{VistDesc}
{\sc Angelo Vistoli}: Grothendieck topologies, fibered categories and descent
  theory.
\newblock In {\em Fundamental algebraic geometry}, volume 123 of {\em Math.
  Surveys Monogr.}, pp. 1--104. Amer. Math. Soc., Providence, RI, 2005.
\newblock [\epfmt{arxiv}{math.AG/0412512}].

\bibitem{Water}\bibhead{Water}
{\sc William~C. Waterhouse}: {\em Introduction to affine group schemes},
  volume~66 of {\em Graduate Texts in Mathematics}.
\newblock Springer-Verlag, New York, 1979.

\bibitem{Weibel}\bibhead{Weibel}
{\sc Charles~A. Weibel}: {\em An introduction to homological algebra},
  volume~38 of {\em Cambridge Studies in Advanced Mathematics}.
\newblock Cambridge University Press, Cambridge, 1994.

\bibitem{Weil}\bibhead{Weil}
{\sc André Weil}: Généralisation des fonctions abéliennes.
\newblock {\em J. Math. Pures Appl., IX. Sér.}, 17:47--87, 1938.

\bibitem{LinWeng}\bibhead{LinWeng}
{\sc Lin Weng}: Geometric arithmetic: a program.
\newblock In {\em Arithmetic geometry and number theory}, volume~1 of {\em Ser.
  Number Theory Appl.}, pp. 211--400. World Sci. Publ., Hackensack, NJ, 2006.

\bibitem{Zoon}\bibhead{Zoon}
{\sc Vincent Zoonekynd}: Th\'eor\`eme de van {K}ampen pour les champs
  alg\'ebriques.
\newblock {\em Ann. Math. Blaise Pascal}, 9(1):101--145, 2002.
\newblock [\epfmt{arxiv}{math.AG/0111073}].

\end{thebibliography}

\end{document}